\documentclass[11pt,a4paper,draft]{article}

\usepackage[a4paper,margin=.96in,bottom=1in,top=.8in]{geometry}
\usepackage{siunitx} % for the degree symbol
\usepackage{mathrsfs} % for the \mathscr font
\usepackage{amsmath,amsthm,amssymb} % standard
\usepackage[final]{hyperref} % standard
\usepackage{enumitem} % used for \label and \ref'ing \items
\usepackage[final]{graphicx}
\usepackage{placeins} % to help wrestle a few floats into position

\numberwithin{equation}{section}
\newtheorem{theorem}{Theorem}[section]
\newtheorem{proposition}[theorem]{Proposition}
\newtheorem{lemma}[theorem]{Lemma}
\newtheorem{corollary}[theorem]{Corollary}
\theoremstyle{definition}
\newtheorem{definition}[theorem]{Definition}
\theoremstyle{plain}
\newtheorem{conjecture}[theorem]{Conjecture}

\newcommand{\R}{\mathbb{R}}
\newcommand{\N}{\mathbb{N}}
\newcommand{\C}{\mathbb{C}}
\newcommand\F{\mathscr F}
\newcommand\soln{\mathscr S}
\newcommand\triv{\mathscr T}
\newcommand\nodal{\mathscr N}
\newcommand\Fg{\mathcal F}
\newcommand\cm{\mathscr C}
\newcommand\cmg{\mathcal C}

\newcommand\T{\mathbb T}
\newcommand\D{\mathbb D}

\renewcommand{\bar}{\overline} 
\newcommand\Ca{\mathcal C} % cauchy integral

\renewcommand       {\Im}  {\operatorname{Im}} % (overriding default)
\DeclareMathOperator{\ran} {\mathrm{ran}}
\renewcommand       {\Re}  {\operatorname{Re}}  % (overriding default)
\newcommand         {\loc} {\mathrm{loc}}
\DeclareMathOperator{\dist}{\mathrm{dist}}
\DeclareMathOperator{\Span}{\mathrm{span}}

\newcommand{\dell}{\partial}
\newcommand{\grad}{\nabla}
\newcommand{\ona}{\textup{~on~}}
\newcommand{\ina}{\textup{~in~}}
\newcommand{\asa}{\textup{~as~}}
\newcommand{\ata}{\textup{~at~}}
\newcommand{\fora}{\textup{~for~}}
\newcommand{\n}[2][]{#1\lVert #2 #1\rVert}
\newcommand{\abs}[2][]{#1\lvert #2 #1\rvert}
\newcommand{\without}{\setminus}
\newcommand{\by}{\times}
\newcommand{\sub}{\subset}
\newcommand{\maps}{\colon}

\title{Global bifurcation of rotating vortex patches}
\author{Zineb Hassainia
  \and
  Nader Masmoudi
  \and 
  Miles H. Wheeler}

\begin{document}
\maketitle

\begin{abstract}
  We rigorously construct continuous curves of rotating vortex patch
  solutions to the two-dimensional Euler equations. The curves are
  large in that, as the parameter tends to infinity, the minimum along
  the interface of the angular fluid velocity in the rotating frame
  becomes arbitrarily small. This is consistent with the conjectured
  existence~\cite{woz:numerical, overman:limiting} of singular
  limiting patches with 90$^\circ$ corners at which the relative fluid
  velocity vanishes. For solutions close to the disk, we prove that
  there are ``Cat's eyes''-type structures in the flow, and provide
  numerical evidence that these structures persist along the entire
  solution curves and are related to the formation of corners. We also
  show, for any rotating vortex patch, that the boundary is analytic
  as soon as it is sufficiently regular.
\end{abstract}

\setcounter{tocdepth}{2}
\tableofcontents

\section{Introduction}
\subsection{Statement of the main results}
We consider the two-dimensional incompressible Euler equations,
written in terms of the vorticity $\omega$ and stream
function $\psi$ as 
\begin{align}
  \label{eqn:omega}
  \dell_t \omega + \grad^\perp \psi \cdot \grad \omega = 0, 
  \qquad 
  -\Delta \psi = \omega.
\end{align}
The fluid velocity is $\grad^\perp\psi = (-\psi_y,\psi_x)$. A vortex
patch is a (weak) solution of \eqref{eqn:omega} with $\omega(z,t) =
1_{D(t)}(z)$ for some simply-connected region $D(t)$. As is typically
done, we restrict to the case where the fluid is at rest at infinity.
We are interested in vortex patches for which, after moving to a
(non-inertial) frame rotating with constant angular velocity $\Omega$,
the region $D$ is stationary. The fluid velocity in the rotating
frame is then $\grad^\perp \Psi$ where the relative stream function
$\Psi = \psi - \tfrac 12 \Omega \abs z^2$ solves 
\begin{subequations} \label{eqn:Psi}
  \begin{align}
    \label{eqn:Psi:lap}
    \Delta \Psi &= 1_D - 2\Omega,\\
    \label{eqn:Psi:asym}
    \grad (\Psi + \tfrac 12 \Omega \abs z^2)&\to 0 \asa \abs z \to \infty,\\
    \label{eqn:Psi:reg}
    \Psi &\in C^1(\C),\\
    \label{eqn:Psi:kin}
    \Psi &= 0 \ona \dell D.
  \end{align}
\end{subequations}
This is a free boundary problem in that the domain $D$ and the
function $\Psi$ are both unknowns. Here and in what follows we
identify $(x,y) \in \R^2$ with $z = x+iy\in \C$ whenever convenient.

Somewhat informally stated, our main existence result for
\eqref{eqn:Psi} is the following.
\begin{theorem}\label{thm:informal}
  For any $m \ge 2$, there exists a continuous curve $\cm$ of rotating
  vortex patches with the symmetries of a regular $m$-gon,
  parametrized by $s \in [0,\infty)$, with the following properties.
  \begin{enumerate}[label=\rm(\alph*)]
  \item\label{thm:informal:bif} {\rm (Bifurcation from the disk)} The
    solution at $s=0$ is the unit disk $D = \D$ rotating with angular
    velocity $\Omega = (m-1)/2m$ and with the angular fluid velocity
    $\dell_r\Psi \equiv 1/2m$ on $\dell D$.
  \item\label{thm:informal:vanish} {\rm (Vanishing angular fluid
    velocity)} As $s \to \infty$,
    \begin{align}
      \label{eqn:vanish}
      \min_{\dell D} \dell_r \Psi \to 0,
    \end{align}
    i.e.~there are points $z(s) \in \dell D(s)$ where the
    angular fluid velocity becomes arbitrarily small.
  \item\label{thm:informal:nodal} {\rm (Monotonicity)}
    For each $s > 0$, the boundary of the patch can be expressed as a polar
    graph $r=R(\theta)$ where $R$ is even, $2\pi/m$-periodic, and
    satisfies
    \begin{align}
      \label{eqn:informal:nodal}
      R'(\theta) < 0 \fora 0 < \theta < \frac \pi m,
      \quad 
      R''(0) < 0,
      \quad 
      R''\Big(\frac \pi m\Big) > 0.
    \end{align}
  \item\label{thm:informal:reg} {\rm (Analyticity)} 
    For each $s \ge 0$, the boundary $\dell D$ (equivalently the
    function $R$ above) is analytic.
  \end{enumerate}
\end{theorem}

\begin{figure}
  \centering
  \includegraphics[scale=1.1]{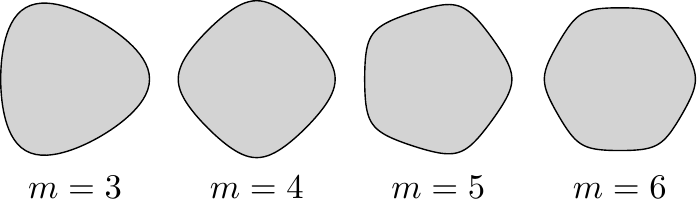} 
  \caption{Patches with various symmetry classes $m$.}
  \label{fig:symmetry}
\end{figure}
\begin{figure} 
  \centering
  \includegraphics[scale=1.1]{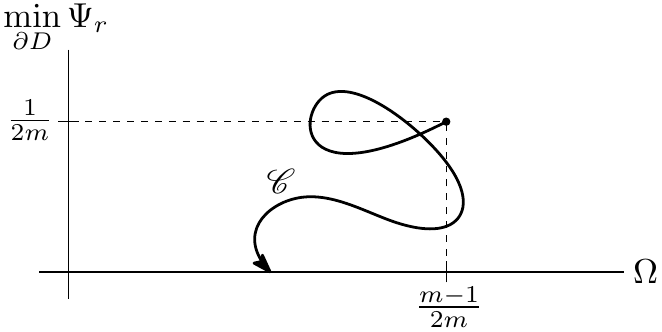} 
  \caption{Sketch of the global bifurcation curve $\cm$ constructed in
  Theorem~\ref{thm:informal}.}
  \label{fig:cmsketch}
\end{figure}

Figure~\ref{fig:symmetry} shows regions $D$ satisfying
\ref{thm:informal:nodal} for various values of $m$. Parts
\ref{thm:informal:bif} and \ref{thm:informal:vanish} of
Theorem~\ref{thm:informal} are illustrated in
Figure~\ref{fig:cmsketch}. Note that the curve $\cm$ is \emph{global}
in that it is not contained in a small neighborhood of its starting
point. Indeed, $\min_{\dell D}\Psi_r = 1/2m$ at the start of $\cm$
while the limiting value is $0$. The analyticity
\ref{thm:informal:reg} is true for any sufficiently smooth solution
satisfying a non-degeneracy condition; see Theorem~\ref{thm:freereg}.
For the precise sense in which $\cm$ is a continuous curve, see
Theorem~\ref{thm:global} and Section~\ref{sec:space}.

\begin{figure}[hb]
  \centering
  \includegraphics[scale=1.1]{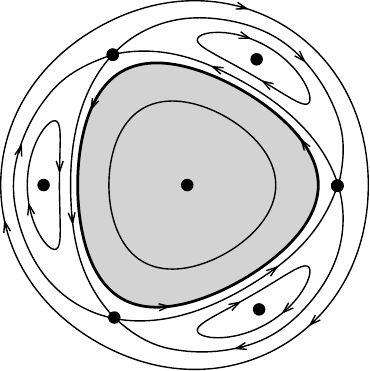} 
  \caption{Phase portrait for the ODE $\dot z = \grad^\perp \Psi$
  showing ``Cat's eyes''.}
  \label{fig:phasealone}
\end{figure}

For $s$ sufficiently small, we also prove the existence of ``Cat's
eye''-type structures in the flow outside of the patch. 
\begin{theorem}[Cat's eyes for small $s$]\label{thm:phase}
  Let $\cm$ be as in Theorem~\ref{thm:informal} and let $s > 0$ be
  sufficiently small. Then the phase portrait of $\dot z = \grad^\perp
  \Psi(z)$ outside of $D$ looks qualitatively like
  Figure~\ref{fig:phasealone}. In particular, there are $m$ saddle points,
  with adjacent saddle points connected by pairs of heteroclinic
  orbits. These heteroclinic orbits enclose regions of periodic orbits
  surrounding $m$ centers. All other orbits are polar graphs $r =
  \tilde r(\theta)$.
\end{theorem}

Based on our above results and the numerical evidence in
Section~\ref{sec:numerical}, we make the following two conjectures:
\begin{conjecture}[Limiting solutions]\label{conj:limiting}
  The singular solutions with \ang{90} corners seen in
  numerics~\cite{woz:numerical,overman:limiting} exist as the (weak)
  limits of patches along $\cm$ as $s \to \infty$.
\end{conjecture}
\begin{conjecture}[Persistence of Cat's eyes]\label{conj:nodal}
  The conclusion of Theorem~\ref{thm:phase} holds for all $s > 0$.
\end{conjecture}
A proof of Conjecture~\ref{conj:limiting} would seem to require, among
a great many other things, a positive resolution of
Conjecture~\ref{conj:nodal}. This is similar to the current state of
the art for steady water waves with constant vorticity;
see~\cite{csv:critical} and the discussion in the next subsection.

\subsection{Historical discussion}
In 1880, Thomson (Lord Kelvin) derived and analyzed the linear
equations for small irrotational disturbances of a three-dimensional
cylindrical vortex~\cite{thomson:columnar}. For purely two-dimensional
disturbances, one finds that a vortex patch with boundary $r =
1+\varepsilon \cos m\theta$ will rotate at constant angular velocity
$\Omega_m = (m-1)/2m$~\cite[Art.~158]{lamb32}. Kirchhoff later
discovered explicit two-dimensional solutions to the full nonlinear
problem in the form of rotating ellipses~\cite[Art.~159]{lamb32}. As
the eccentricity vanishes, the angular velocity of these ellipses
approaches Kelvin's $\Omega_2$.

The first rigorous existence proof for nonlinear rotating vortex
patches for $m \ge 3$ is due to Burbea in 1982~\cite{burbea:motions}.
Reformulating the problem in terms of a conformal mapping, he used the
celebrated Crandall--Rabinowitz theorem~\cite{cr:simple} on
bifurcation from a simple eigenvalue and for each $m$ obtained a small
curve of solutions close to the unit disk. In 2013, Hmidi, Mateu and
Verdera~\cite{hmv:reg} again used Crandall--Rabinowitz methods to
construct local curves of solutions, this time showing that the
boundaries $\dell D$ are smooth. This regularity result was further
improved by Castro, C\'ordoba, and G\'omez-Serrano~\cite{ccg:reg} who
in 2016 constructed a local curve of solutions with $\dell D$
analytic.

In recent years there has been a burst of rigorous mathematical work
on rotating vortex patches and related problems. In addition to the
results mentioned above, there are existence proofs for rotating
patches close to Kirchhoff's ellipses~\cite{ccg:reg,hm:ellipse}, pairs
of vortex patches~\cite{hm:pairs}, multiply-connected
patches~\cite{hhmv:doubly}, and patches in bounded
domains~\cite{hhhm:disc}. Many of these results apply not only to the
Euler equations but also to the inviscid Surface Quasi-Geostrophic
equations or the generalized Surface Quasi-Geostrophic equations; in
this context also see \cite{ccgz:geometric,ccg:sqg}.

It is important to emphasize that all of the above analytical results
treat patches which are sufficiently close either to the unit disk or
to some other explicit solution. Numerically, however, solutions have
been found far beyond these perturbative regimes. In 1978, Deem and
Zabusky~\cite{dz:vstates} found branches of rotating patches (which
they called ``V-states'') with different symmetry classes $m$
bifurcating from the unit disk. Wu, Overman, and
Zabusky~\cite{woz:numerical} went further along the same branches in
1984 and found singular limiting solutions with \ang{90} corners; see
Figure~\ref{fig:large} on page~\pageref{fig:large}.
Overman~\cite{overman:limiting} then performed a careful asymptotic
analysis near the corner of a hypothetical vortex patch (satisfying
several assumptions), and confirmed analytically that either the
corner is a cusp with an interior angle of 0, or that the interior
angle is \ang{90} as seen in the numerics. Patches bifurcating from
Kirchhoff ellipses rather than the unit disk were first computed by
Kamm in his thesis~\cite{kamm:thesis}. The papers
\cite{hhmv:doubly,hhhm:disc} mentioned in the previous paragraph also
contain numerical results on doubly-connected vortex patches and
on patches in a bounded domain.

To our knowledge, Theorem~\ref{thm:informal} is the first
existence proof for rotating vortex patches which is global in the
sense that it is not limited to a small neighborhood of an explicit
solution. 
Our methods are inspired by global results for steady water waves, and
in particular the real-analytic bifurcation techniques in
\cite{bt:analytic}. For steady water waves, there is an analogue of
Conjecture~\ref{conj:limiting} known as the ``Stokes conjecture''. In
the absence of vorticity, it was famously proven in a serious of
papers culminating in \cite{aft}. When Cat's eyes are permitted in the
flow, however, the existence and nature of limiting solutions remains
an important open problem; see \cite{csv:critical}.

Before continuing to the outline, we lastly compare our results to the
variational work of Turkington~\cite{turkington:steady,
turkington:nfold} in the 1980s. In \cite{turkington:steady},
Turkington considered \emph{steady}, \emph{non-rotating}  vortex
patches in a bounded domain, and in particular the singular limit as
the patches become point vortices. This non-rotating problem is
fundamentally different from ours: the flow no longer has Cat's eyes
and the patch is simply expressed as $D= \{\Psi > 0\}$. That being
said, Turkington's result is indeed global in the sense that it
constructs patches with any prescribed area (less than the area of the
bounded domain). The regularity of the solutions, outside of the
scaling limit mentioned above, is left open, and so it is possible
that some of these patches have singular boundaries. In
\cite{turkington:nfold}, Turkington considered an unbounded fluid
domain with $N$ symmetrically arranged vortex patches rotating about
the origin. Restricting attention to a fixed region about each patch
$D$, he first solved a modified variational problem for which again
$D=\{ \Psi > 0\}$, but was only able to guarantee that this yields a
solution to the full problem in the limit as the patches approach
point vortices.

\subsection{Outline of the proof}

As in \cite{burbea:motions}, we reformulate \eqref{eqn:Psi} in terms
of a conformal mapping $\Phi$. Let $\D \sub \C$ denote the unit disk
with boundary $\T := \dell \D$, and assume that the vortex patch $D
\sub \C$ is a bounded and simply-connected $C^{k+\alpha}$ domain for
some integer $k \ge 1$ and $\alpha \in (0,1)$. A conformal mapping
$\Phi$ from $\C \without \overline\D$ to $\C \without \overline D$
will then extend to a $C^{k+\alpha}$ mapping $\C \without \D \to \C
\without D$, and $\phi := \Phi|_\T$ will give a $C^{k+\alpha}$
parametrization of $\dell D$ \cite[Theorem~3.6]{pommerenke:book}. 

With $D$ fixed, the unique solution $\Psi$ of
\eqref{eqn:Psi:lap}--\eqref{eqn:Psi:reg} can be written explicitly in
terms of the Newtonian potential of $1_D$. Differentiation and an
application of Green's theorem then yields
\begin{align}
  \label{eqn:gradPsi}
  -\bar{\grad \Psi(\Phi(w))} 
  =
  \Omega\bar{ \Phi(w) }
  + \frac 1{4\pi i} \int_\T 
  \frac{\bar{\Phi(\tau)}-\bar{\Phi(w)}} { \Phi(\tau) -
  \Phi(w)} \Phi'( \tau)\, d\tau,
\end{align}
where here $\grad\Psi = \Psi_x + i\Psi_y$. Differentiating the
remaining equation \eqref{eqn:Psi:kin} along the boundary and plugging
in \eqref{eqn:gradPsi}, one sees that \eqref{eqn:Psi} is equivalent
to the integral equation
\begin{align}
  \label{eqn:phi}
  \Im \left\{
    \left(
      \Omega\bar{ \phi(w) }
      + \frac 1{4\pi i} \int_\T 
      \frac{\bar{\phi(\tau)}-\bar{\phi(w)}} { \phi(\tau) -
      \phi(w)} \phi'( \tau)\, d\tau 
    \right)
    w \phi'(w)
  \right\} = 0
\end{align}
for the restriction $\phi = \Phi|_\T$. See Lemma~\ref{lem:stream}
below for more details of the equivalence, and also for instance \cite{hmv:reg}.

Two famous objects from complex analysis appear in \eqref{eqn:phi}.
The first is the Cauchy integral operator
\begin{align}
  \label{eqn:cauchy}
  \Ca(\phi) \maps
  g \mapsto
  \frac 1{2\pi i}\int_\T 
  \frac{g(\tau)-g(w)} { \phi(\tau) -
  \phi(w)} \phi'( \tau)\, d\tau
\end{align}
associated to the curve $\dell D = \phi(\T)$. By a result of Lanza
de Cristoforis and Preciso~\cite{ldcl:cauchy}, this bounded operator
$C^{k+\alpha}(\T) \to C^{k+\alpha}(\T)$ depends real-analytically on
$\phi$ in an open subset of $C^{k+\alpha}(\T)$; see
Theorem~\ref{thm:cauchy} below. This analyticity seems not to have
been previously taken advantage of in the mathematical literature on
rotating vortex patches. In addition to enabling us to skip tedious
verifications of the regularity of the dependence of various
expressions on $\phi$, it enables us to use a powerful global
bifurcation theory specialized to analytic operators
\cite{dancer:global,bt:analytic}. 

Secondly, the equation \eqref{eqn:phi} itself can be thought of as a
of ``quasilinear Riemann--Hilbert problem''. This is a useful approach
to \eqref{eqn:phi} which appears to be new. Setting 
\begin{align*}
  % \label{eqn:V}
  A := 
  (
  \Omega\bar{ \phi }
  + \tfrac 12 \Ca(\phi) \bar \phi
  ) w,
\end{align*}
\eqref{eqn:phi} takes the form
\begin{align}
  \label{eqn:rh}
  \Im\{A \phi'\} = 0.
\end{align}
Thanks to the mapping properties of $\Ca$ mentioned above, the
``coefficient'' $A$ has the same regularity as $\phi$, while the
derivative $\phi'$ appears linearly. Using standard formulas from the
theory of Riemann--Hilbert problems, we can explicitly invert
\eqref{eqn:rh} to find $\phi'$ in terms of $A$. This is helpful
because of the good regularity properties of $A$ and also its close
connection to the fluid velocity via \eqref{eqn:gradPsi}. 

We remark that, by \eqref{eqn:gradPsi}, $A \in C^{k+\alpha}$ implies
that the composition $\grad\Psi \circ \phi$ is also $C^{k+\alpha}$.
This is an improvement over naive Schauder estimates based on the
elliptic equation \eqref{eqn:Psi}, which only give $\grad\Psi \in
C^{k-1+\alpha}$. There are of course other ways to obtain similar gains in
regularity without recourse to conformal mappings; see for instance
the ``regularizing diffeomorphisms'' in
\cite[Section~2.2.2]{lannes:book}. 

The outline of the paper is as follows. In Section~\ref{sec:prelim},
we collect several preliminary results and reformulate \eqref{eqn:phi}
as a nonlinear operator equation $\F(\phi-w,\Omega) = 0$. Here $\F$ is
an analytic operator defined on an open subset $U$ of a Banach space
$X$ which encodes the $m$-fold symmetry. In
Section~\ref{sec:localbif}, we study the set of solutions to this
equation when $\n{\phi - w}_{C^{3+\alpha}}$ is small. It consists of
the ``trivial'' line of solutions with $\phi(w) \equiv w$ together
with a sequence of analytic curves (i.e.~curves with analytic
parametrizations) bifurcating from this axis at discrete frequencies
$\Omega_{nm}$. The existence of these curves is well known, but the
analyticity appears to be new. In Section~\ref{sec:global}, we
construct a global curve $\cm$ of solutions using the analytic global
bifurcation theory of Dancer \cite{dancer:global} and Buffoni--Toland
\cite{bt:analytic}. This curve either \ref{thm:global:loop}
is a closed loop or \ref{thm:global:blowup} ``blows up'' in a certain
sense as the parameter $s\to \infty$. In Section~\ref{sec:nodal}, we
show that alternative \ref{thm:global:loop} cannot happen by tracking
certain ``nodal properties'' related to \eqref{eqn:informal:nodal}
using maximum principle and continuation arguments. We also prove
Theorem~\ref{thm:phase} on the streamlines of solutions with small
$s$. In Section~\ref{sec:n:reg}, we show that $\dell D$ is analytic
for every solution in $\cm$ by using a result of Kinderlehrer,
Nirenberg, and Spruck~\cite{kns:freereg}. In
Section~\ref{sec:uniform}, we turn to the second alternative
\ref{thm:global:blowup} and complete the proof of
Theorem~\ref{thm:informal}. Finally, in Section~\ref{sec:numerical},
we provide numerical evidence for Conjectures~\ref{conj:limiting} and
\ref{conj:nodal}. Appendix~\ref{sec:rh} contains some needed facts
about linear Riemann--Hilbert problems, and Appendix~\ref{sec:stream}
contains several identities relating derivatives of the stream
function $\Psi$ to derivatives of $\phi$.

\section{Formulation and preliminaries}\label{sec:prelim}

In the rest of the paper we fix the integer $m \ge 2$ describing the
symmetry class of the solutions under consideration.

\subsection{Some useful formulas}

The following lemma allows us to switch between the two formulations
\eqref{eqn:phi} and \eqref{eqn:Psi} of the rotating vortex patch
problem. While \eqref{eqn:phi} is more convenient
functional-analytically, the elliptic equation \eqref{eqn:Psi} will
often be more useful for establishing qualitative results. To simplify
the formulas, we make use of the Wirtinger derivatives
\begin{align*}
  \dell_z := \tfrac 12 ( \dell_x - i\dell_y),
  \qquad 
  \dell_{\bar z} := \tfrac 12 ( \dell_x + i\dell_y).
\end{align*}
\begin{lemma}[Partials of the stream function]\label{lem:stream}
  Suppose $D \sub \C$ is a bounded and simply-connected $C^{k+\beta}$
  domain for some $k \ge 3$, and let $\Psi \in C^{k+\beta}(\overline
  D) \cap C^{k+\beta}(\C \without D) \cap C^1(\C)$ be the unique (up
  to an additive constant) solution to
  \eqref{eqn:Psi:lap}--\eqref{eqn:Psi:reg}. Let $\Phi \maps \C
  \without \D \to \C \without D$ be a conformal map, with
  $C^{k+\beta}$ restriction $\phi := \Phi|_\T$.
  \begin{enumerate}[label=\rm(\alph*)]
  \item On $\dell D$, the partials of $\Psi^- := \Psi|_{\C \without D}$ are given by
    \begin{align} \label{eqn:Psizw}
      \begin{aligned}
        % \label{eqn:Psizw:z}
        (\dell_z \Psi) \circ \phi 
        &= \tfrac 14 \Ca(\phi)\bar\phi - \tfrac \Omega 2 \bar{\phi},\\
        % \label{eqn:Psizw:zz}
        (\dell_z^2 \Psi^-) \circ \phi 
        &= \tfrac 14 \Ca(\phi) F_2(\phi) ,\\
        % \label{eqn:Psizw:zzz}
        (\dell_z^3\Psi^-) \circ \phi
        &= \tfrac 14\Ca(\phi) F_3(\phi),
      \end{aligned}
    \end{align}
    where
    \begin{align*}
      % \label{eqn:F2F3}
      F_2(\phi) := \frac{\bar{\phi'}}{w^2\phi'},
      \qquad 
      F_3(\phi) := -
      \frac{2\bar{\phi'}}{w^3(\phi')^2}
      - \frac{\bar{\phi''}}{w^4(\phi')^2}
      - \frac{\bar{\phi'} \phi''}{w^2 (\phi')^3}.
    \end{align*}
  \item In particular, (a constant shift of) $\Psi$ solves
    \eqref{eqn:Psi:kin} if and only if $\phi$ satisfies
    \eqref{eqn:phi}.

  \end{enumerate}
  \begin{proof}
    The representation of $\dell_z\Psi$ and the equivalence are
    implicit in earlier work on rotating vortex patches,
    e.g.~\cite{burbea:motions,hmv:reg}. The extension to higher partials
    involves an integration-by-parts argument to eliminate
    singular terms in the integral (e.g.~\cite[\S4.4]{gakhov}). For
    the reader's convenience the details are provided in
    Appendix~\ref{sec:stream}.
  \end{proof}
\end{lemma}

It will also often be useful to express $\phi$ in polar coordinates as
$\phi = \rho e^{i\vartheta}$. For ease of reference, we record the
formulas for $\rho'$ and $\vartheta'$ in the following lemma.
\begin{lemma}[$\phi$ in polar coordinates]\label{lem:polar}
  For any $\phi\in C^{k+\beta}(\T,\C \without \{0\})$ we can write
  \begin{align}
    \label{eqn:polar}
    \phi(e^{it}) = \rho(t) e^{i\vartheta(t)}
    \fora t \in \R,
  \end{align}
  where $\rho > 0$ and $\vartheta$ are periodic real-valued
  $C^{k+\beta}$ functions. Their first derivatives are
  \begin{align}
    \label{eqn:polarder}
    \vartheta'(t) = 
    \Re \left( \frac{e^{it}\phi'(e^{it})}{\phi(e^{it})} \right),
    \qquad 
    \rho'(t) = -\rho(t) \Im 
    \left( \frac{e^{it}\phi'(e^{it})}{\phi(e^{it})} \right).
  \end{align}
  \begin{proof}
    The existence of $\rho,\vartheta$ is standard.
    Differentiating \eqref{eqn:polar} with respect to $t$ and dividing
    by $\phi$ yields
    \begin{align*}
      \frac{ie^{it}\phi'(e^{it})}{\phi(e^{it})} = \frac{\rho'(t)}{\rho(t)}
      + i\vartheta'(t),
    \end{align*}
    which has \eqref{eqn:polarder} as real and imaginary parts.
  \end{proof}
\end{lemma}

\subsection{Function spaces and open sets}\label{sec:space}

Let $k \ge 1$ be an integer and $\beta \in (0,1)$. 

Observe that
\eqref{eqn:phi} has a scaling symmetry: if $\phi$ is a solution then
so is $\lambda\phi$ for any $\lambda\ne 0$. Following
\cite{burbea:motions}, we will kill this symmetry by fixing the
coefficient of the linear term in the Laurent expansion of $\phi$ to
be $1$. We then think of \eqref{eqn:phi} as an equation for the
remainder 
\begin{align*}
  % \label{eqn:f}
  f(w) := \phi(w) -w,
\end{align*}
which we require to lie in the Banach space
\begin{align*}
  X^{k+\beta} &=  \left\{ f \in C^{k+\beta}(\T) : 
  f(w) = \sum_{n=1}^\infty \frac{a_n}{w^{nm-1}},
  \ a_n \in \R \right\}.
\end{align*}
The absence of $w^p$ terms in the Fourier series of $f$ for $p \ge 1$
guarantees that $f$ extends to a holomorphic function $F$ on $\C
\without \D$ with $F'(\infty) = 0$. Note, however, that there is as
of yet no guarantee that $\Phi(w):= w+F(w)$ will be conformal.
\begin{subequations}\label{eqn:sym}
  The absence of the other terms in the series is equivalent to the
  discrete rotation symmetry 
  \begin{align}
    \label{eqn:sym:rot}
    \phi(e^{2\pi i/m} w) = e^{2\pi i/m} \phi(w),
  \end{align}
  while the fact that the coefficients are real is equivalent to
  the reflection symmetry 
  \begin{align}
    \label{eqn:sym:ref}
    \phi(\bar w) = \bar{\phi(w)}.
  \end{align}
\end{subequations}
We also define another Banach space $Y^{k-1+\beta}$, which will be the space for the
right hand side of \eqref{eqn:phi},
\begin{align*}
  Y^{k-1+\beta} &=  \left\{
    h \in C^{k-1+\beta}(\T) : h(e^{2\pi i/m}w) = h(w),\,
  h(\bar w) = -\bar {h(w)}
  \right\}.
\end{align*}

We will not work in $X^{k+\beta}$ directly but in a convenient open
subset $U^{k+\beta} = U_1^{k+\beta} \cap U_2^{k+\beta} \cap
U_3^{k+\beta}$ where
\begin{align*}
  U_1^{k+\beta} &= \left\{ (\phi-w,\Omega) \in X^{k+\beta} \by \R : \inf_{\tau \ne w}, 
  \left|\frac{\phi(\tau)-\phi(w)}{\tau - w}\right| > 0 \right\},\\
  U_2^{k+\beta} &= \left\{ (\phi-w,\Omega) \in U_1^{k+\beta} : 
  \abs[\big]{\Omega\bar{ \phi } + \tfrac 12 \Ca(\phi) \bar\phi}
    > 0 \right\},\\
  U_3^{k+\beta} &= \left\{ (\phi-w,\Omega) \in X^{k+\beta} \by \R :
  \Re\left(\frac{w\phi'}{\phi}\right) > 0 \ona \T \right\}. 
\end{align*}

\subsubsection{Open set for the Cauchy integral.}
The set $U_1^{k+\beta}$ is chosen so that the Cauchy
integral operator $\Ca(\phi)$ appearing in \eqref{eqn:phi} is well-behaved:
\begin{theorem}[Analyticity of the Cauchy integral \cite{ldcl:cauchy}]
  \label{thm:cauchy}
  \hfill
  \begin{enumerate}[label=\rm(\alph*)]
  \item The set $\tilde U_1^{k+\beta}$ defined by 
    \begin{align*}
      \tilde U_1^{k+\beta} = \left\{ \phi \in C^{k+\beta}(\T) : \inf_{\tau \ne w}
      \left|\frac{\phi(\tau)-\phi(w)}{\tau-w}\right| > 0 \right\} 
    \end{align*}
    is an open subset of $C^{k+\beta}(\T)$. Moreover $\phi \in
    C^{k+\beta}(\T)$ lies in $\tilde U_1^{k+\beta}$ if and only if $\phi$ is
    injective and $\phi' \ne 0$.
  \item The formula \eqref{eqn:cauchy} describes a real-analytic mapping
    \begin{align*}
      \Ca \maps \tilde U_1^{k+\beta} 
      \to \mathscr L(C^{k+\beta}(\T)),
    \end{align*}
    where $\mathscr L(C^{k+\beta}(\T))$ is the Banach space of
    bounded linear maps from $C^{k+\beta}(\T)$ to itself.
  \end{enumerate}
\end{theorem}
Thus $(\phi-w,\Omega) \in X^{k+\beta} \by \R$ lies in $U_1^{k+\beta}$
if and only if $\phi$ is injective and satisfies $\phi' \ne 0$. This
injectivity also guarantees that $\phi$ can indeed be extended to a
conformal map $\Phi$ on $\C \without
\D$~\cite[p.\,16]{pommerenke:book}. Note that while we will always
apply Theorem~\ref{thm:cauchy} to mappings $\phi$ which extend to
holomorphic functions, the theorem itself makes no such restriction. 

\subsubsection{Open set for Riemann--Hilbert problems.}
From \eqref{eqn:Psizw} we see that  $(f,\Omega) \in U_2^{k+\beta}$ if
and only if the relative fluid velocity $\grad\Psi$ does not vanish on
$\dell D$. This also guarantees that the coefficient $A$ multiplying
$\phi'$ in the Riemann--Hilbert problem \eqref{eqn:rh} is
non-vanishing, which will eventually enable us to apply
Lemma~\ref{lem:rh} below with $a=A$.

\begin{lemma}[Linear Riemann--Hilbert problems]\label{lem:rh}
  Suppose that $a \in C^{k-1+\beta}(\T,\C)$ has winding number $0$ in
  that
  \begin{align*}
    % \label{eqn:nowind}
    \abs a > 0 \quad \textup{ and } \quad 
    \arg a(e^{it})\Big|^{t=2\pi}_{t=0} = 0,
  \end{align*}
  and also that $a$ has the symmetry properties $a(\bar w) =
  \bar{a(w)}$ and $a(e^{2\pi i/m}w) = a(w)$. Then:
  \begin{enumerate}[label=\rm(\alph*)]
  \item \label{lem:rh:fund} 
    The problem 
    \begin{align}
      \label{eqn:rh:fund}
      \Im\{ ag' \} = 0 \ona \T,
      \qquad g-w \in X^{k+\beta}
    \end{align}
    has a unique solution $g=g_0$, whose derivative is given
    explicitly by
    \begin{align*}
      % \label{eqn:rh:fundsoln}
      g'_0(w) = 
        \exp \left\{ 
        \frac w{2\pi} \int_\T 
        \frac{ \tau^{-1} \theta(\tau) -w^{-1} \theta(w) }{\tau -w}
        \, d\tau
        \right\},
    \end{align*}
    where here
    \begin{align*}
      \theta(w) = \arg\bigg(\frac{a(w)}{\bar{a(w)}}\bigg)
    \end{align*}
    and the branch of the $\arg$ function is fixed by requiring $\theta(1) = 0$.
  \item \label{lem:rh:inv} The operator 
    \begin{align*}
      L \maps X^{k+\beta} \to Y^{k-1+\beta},
      \qquad g \mapsto \Im\{Ag'\}
    \end{align*}
    is well-defined and invertible, with inverse operator
    characterized by
    \begin{align}
      \label{eqn:rh:inv}
      \frac d{dw} L^{-1} h(w) 
      =
        -\frac{wg_0'(w)}\pi
        \int_\T \frac 1{\tau - w}
        \left( \frac {h(\tau)}{a(\tau) g_0'(\tau) \tau}
        - \frac {h(w)}{a(w) g_0'(w) w}\right) d\tau.
    \end{align}
  \end{enumerate}
  \begin{proof}
    We postpone the proof, which relies on classical results for
    Riemann--Hilbert problems (e.g. \cite{muskhelishvili}), to
    Appendix~\ref{sec:rh}.
  \end{proof}
\end{lemma}

\subsubsection{Open set for graphical boundary.}
The definition of the final open set $U_3^{k+\beta}$ has a geometric
interpretation: letting $\phi(e^{it}) = \rho(t) e^{i\vartheta(t)}$ as
in Lemma~\ref{lem:polar}, we see that $(\phi-w,\Omega) \in
U_3^{k+\beta}$ is equivalent to $\vartheta' > 0$. Thus $\phi(\T)$ is a
polar graph $r=R(\theta)$ for some $C^{k+\beta}$ function $R$. The
numerical evidence~\cite{woz:numerical} suggests that this is the case
for all but the limiting solution, which is still graphical but loses
regularity. For solutions in $U_2^{k+\beta}$, we will see that
membership in $U_3^{k+\beta}$ is equivalent to the nonvanishing of the
relative angular fluid velocity $\dell_r \Psi$ on $\dell D$.

One useful feature of the set $U_3^{k+\beta}$ is that it is completely
contained in $U_1^{k+\beta}$, guaranteeing that the Cauchy
integral operator is analytic.
\begin{lemma}\label{lem:u1u3}
  $U_3^{k+\beta} \sub U_1^{k+\beta}$.
  \begin{proof}
    Let $(\phi-w,\Omega) \in U_1^{k+\beta}$. By Lemma~2.5 in
    \cite{ldcl:cauchy} it suffices to show that $\phi$ is injective
    and that $\phi' \ne 0$. That $\phi' \ne 0$ follows
    immediately from the definition of $U_3^{k+\beta}$. To see that $\phi|_\T$
    is injective, we simply compute that
    \begin{align*}
      \Im \log \frac{\phi(e^{is_1})}{\phi(e^{is_2})}
      = \Im \int_{s_2}^{s_1} \frac d{dt}
      \log \phi(e^{it})\, dt
      = \int_{s_2}^{s_1} \Re \frac{e^{it} \phi'(e^{it})}{\phi(e^{it})}\,
      dt 
      \ne 0
    \end{align*}
    whenever $s_1 \ne s_2$.
  \end{proof}
\end{lemma}

\subsection{Nonlinear operator and solution set}

Setting
\begin{align*}
  U^{k+\beta} = U_1^{k+\beta} \cap U_2^{k+\beta} \cap U_3^{k+\beta} = U_2^{k+\beta} \cap U_3^{k+\beta},
\end{align*}
we define our nonlinear operator
\begin{align*}
  \F^{k+\beta} \maps U^{k+\beta} \to Y^{k-1+\beta} 
\end{align*}
by simply rewriting the left hand side of \eqref{eqn:phi}:
\begin{align}
  \label{eqn:F}
  \F^{k+\beta}(f,\Omega) := 
  \Im \left\{
    \big(
     \Omega(\bar w + \bar f)
      + \tfrac 12 \Ca(w+f) (\bar w + \bar f)
      \big)
    w (1+f')
  \right\}.
\end{align}
\begin{lemma}\label{lem:analytic}
  The map $\F^{k+\beta} \maps U^{k+\beta} \to Y^{k-1+\beta}$ is
  well-defined and analytic. Moreover, for $(f,\Omega) \in
  U^{k+\beta}$, $B=\Ca(w+f)(\bar w + \bar f)$ enjoys the symmetry
  properties $B(\bar w) = \bar{B(w)}$ and $B(e^{2\pi i/m} w) =
  e^{-2\pi i/m} B(w)$.
  \begin{proof}
    The analyticity of $\F^{k+\beta} \maps U^{k+\beta} \to C^{k+\beta}(\T)$ is
    an immediate consequence of Theorem~\ref{thm:cauchy} and the
    inclusion $U_1^{k+\beta} \sub U^{k+\beta}$. The symmetry
    properties follow from straightforward manipulations using the
    identities $\phi(e^{2\pi i/m}w) = e^{2\pi i/m}\phi(w)$,
    $\phi'(e^{2\pi i/m} w) = \phi'(w)$, and $\phi(\bar w) =
    \bar{\phi(w)}$ for $\phi = w + f$.
  \end{proof}
\end{lemma}

We easily calculate that, for any $\Omega \in \R$,
\begin{align*}
  \F^{k+\beta}(0,\Omega)
  = 
  \Im \left\{
    \left(
      \Omega\bar w
      + \frac 1{4\pi i} \int_\T 
      \frac{\bar\tau-\bar w} { \tau -
      w} \, d\tau 
    \right)
    w 
  \right\}
  = \Im \{ \Omega\abs w^2 - \tfrac 12 \} = 0,
\end{align*}
corresponding to the fact that the unit disc $D=\D$ is a rotating
vortex patch with angular velocity $\Omega$. We call these ``trivial''
solutions and introduce the following notation.

\begin{definition}[Trivial solutions, solution set]
  \label{def:triv}
  The set of ``trivial'' solutions is denoted by
  \begin{align*}
    % \label{eqn:Tdef}
    \triv^{k+\beta} = \{ (0,\Omega) : \Omega \in \R \} \sub
    U^{k+\beta},
  \end{align*}
  and the full solution set by
  \begin{align*}
    % \label{eqn:Sdef}
    \soln^{k+\beta} = \{ (f,\Omega) \in U^{k+\beta} : \F^{k+\beta}(f,\Omega) = 0 \}.
  \end{align*}
\end{definition}

While in the above discussion $k \ge 1$ and $\beta \in (0,1)$ have
been arbitrary, in what follows we will for the most part fix $k = 3$
and $\beta = \alpha \in (0,1)$. To simplify notation, we therefore
introduce the abbreviations
\begin{gather*}
  X := X^{3+\alpha},
  \quad
  Y := Y^{2+\alpha},
  \quad
  U := U^{3+\alpha},
  \quad
  U_1 := U_1^{3+\alpha},
  \quad
  U_2 := U_2^{3+\alpha},
  \quad
  U_3 := U_3^{3+\alpha},\\
  \quad
  \F := \F^{3+\alpha},
  \quad
  \soln := \soln^{3+\alpha},
  \quad
  \triv := \triv^{3+\alpha}.
\end{gather*}

\section{Local bifurcation}\label{sec:localbif}

In this section we describe the solution set $\soln$ near the axis
$\triv$ of trivial solutions. 
The main tools are the implicit function theorem and the following
analytic version of the classical Crandall--Rabinowitz theorem
\cite{cr:simple}.
\begin{theorem}[Theorem 8.3.1 in \cite{bt:analytic}]
  \label{thm:genlocal}
  Let $X,Y$ be real Banach spaces, $U \sub X \by \R$ an open set, and $\Fg \maps U
  \to Y$ a real-analytic function. Suppose that
  \begin{enumerate}[label=\rm(\alph*)]
  \item $\Fg(0,\lambda) = 0$ for all $\lambda$ in a neighborhood of
    $\lambda_0 \in \R$;
  \item %\label{thm:genlocal:kernel} 
    $\Fg_x(0,\lambda_0)$ is a Fredholm operator of index zero, with
    a one-dimensional kernel spanned by $\xi_0 \in X$; and
  \item %\label{thm:genlocal:trans} 
    the ``transversality condition'' $\Fg_{x\lambda}(0,\lambda_0)\xi_0
    \notin \ran \Fg_x(0,\lambda_0)$ holds.
  \end{enumerate}
  Then $(0,\lambda_0)$ is a bifurcation point in the following sense.
  There exists $\varepsilon > 0$ and a pair of analytic functions
  $(\tilde x,\tilde \lambda) \maps (-\varepsilon,\varepsilon) \to \R
  \by U$ such that
  \begin{enumerate}[label=\rm(\roman*)]
  \item $\Fg(\tilde x(s),\tilde \lambda(s)) = 0$ for $s \in
    (-\varepsilon,\varepsilon)$;
  \item $\tilde x(0) =0$,  $\tilde \lambda(0) = \lambda_0$, and $\tilde x'(0) = \xi_0$; and
  \item there exists an open neighborhood $V$ of $(0,\lambda_0)$ in
    $\R \by X$ such that
    \begin{align*}
      \big\{(x,\lambda) \in V : \Fg(x,\lambda) = 0,\, x \ne 0\big\} =  
      \big\{(\tilde x(s),\tilde\lambda(s)) : 0 < \abs s < \varepsilon \big\}.
    \end{align*}
  \end{enumerate}
\end{theorem}

With Theorem~\ref{thm:genlocal} in mind, we next calculate
$\F_f(0,\Omega)$.
\begin{lemma}\label{lem:Fdiff}
  The Fr\'echet derivative $\F_f(0,\Omega)$ is given by
  \begin{align}
    \label{eqn:Fdiff}
    \F_f(0,\Omega) g  &= 
    \Im \Big\{ \big(\Omega+\tfrac 12 w\big)g' + \Omega w \bar g \Big\}.
  \end{align}
  \begin{proof}
    Straightforward differentiation gives
    \begin{align*}
      \F_f(0,\Omega) g &= \Im\Bigg\{\bigg(
      \Omega\big[\bar w{g'(w)}+\bar{g(w)}\big]
      + \,{{g'(w)}}\frac1{4\pi i}\int_\T\frac{\bar\xi-\bar w}{\xi-w}d{\xi}
      + \frac1{4\pi i}\int_\T\frac{\bar\xi-\bar w}{\xi-w}g'(\xi)d\xi\notag\\ 
      & \qquad\qquad\qquad
      +\frac1{4\pi i}\int_\T\frac{\bar{g(\xi)}-\bar{g(w)}}{\xi-w}d\xi
      -\frac1{4\pi i}\int_\T\frac{(g(\xi)-g(w))(\bar\xi-\bar w)}{(\xi-w)^2}d\xi
      \bigg)w\Bigg\}.
    \end{align*} 
    The first integral is easy to compute:
    \begin{align*}
      \int_\T \frac{\bar \xi - \bar w}{\xi - w} \, d\xi
      =  \int_\T 
      \frac{ \frac 1\xi - \frac 1w }{\xi - w}\, d\xi
      = - \frac 1w \int_\T \frac {d\xi}\xi
      = - \frac {2\pi i}w,
    \end{align*}
    where here we've used the identity $\bar\xi = 1/\xi$ for $\xi
    \in \T$. It therefore suffices to show that the remaining
    integrals vanish. Letting $G$ be the holomorphic extension
    of $g$, we obtain through similar manipulations that
    \begin{align*}
      \int_\T\frac{\bar{g(\xi)}-\bar{g(w)}}{\xi-w}d\xi 
      = 
      \int_\T \frac{\bar\xi - \bar w}{\xi -
      w}g'(\xi)\,d\xi 
      = \int_\T\frac{(g(\xi)-g(w))(\bar\xi-\bar w)}{(\xi-w)^2}d\xi
      =
      -\frac {2\pi i}w G'(\infty) = 0,
    \end{align*}
    where $G'(\infty)$ vanishes thanks to $g \in X$. 
  \end{proof}
\end{lemma}

Applying Theorem~\ref{thm:genlocal} together with the implicit
function theorem in our setting we obtain the following, where 
\begin{align*}
  % \label{eqn:Omega}
  \Omega_{nm} := \frac{nm-1}{2nm}, 
  \qquad n=1,2,3,\ldots
\end{align*}
are the critical frequencies from Kelvin's linear analysis, and 
the trivial solution set $\triv$ and full solution set
$\soln$ were introduced in Definition~\ref{def:triv} and the following
paragraph; see Figure~\ref{fig:localbif} for an illustration.
\begin{figure} 
  \centering
  \includegraphics[scale=1.1]{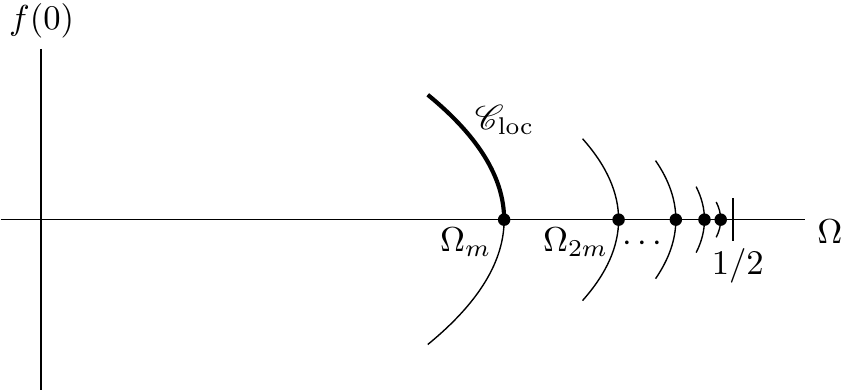} 
  \caption{The local curves constructed in Theorem~\ref{thm:locbif}
  near $\Omega=\Omega_m,\Omega_{2m},\ldots$. The portion in bold is
  the curve $\cm_\loc$ defined in Definition~\ref{def:loc}.}
  \label{fig:localbif}
\end{figure}
\begin{theorem}[Local structure]\label{thm:locbif}
  Fix $\Omega \in \R$. 
  \begin{enumerate}[label=\rm(\roman*)]
  \item \label{thm:locbif:imp} 
    \textup{(No bifurcation)}
    If $\Omega$ is not one of the $\Omega_{nm}$, then there is a
    neighborhood $V$ of $(0,\Omega)$ in $U$ such that
    $\soln \cap V \sub \triv$.
  \item \label{thm:locbif:bif} 
    \textup{(Bifurcation)}
    For every $n$, 
    $(0,\Omega_{nm})$ is a bifurcation point in the following sense.
    There exists $\varepsilon > 0$ and a pair of analytic functions
    $(\tilde f,\tilde \Omega) \maps (-\varepsilon,\varepsilon) \to 
    U$ with the following properties:
    \begin{enumerate}[label=\rm(\alph*)]
    \item $\F(\tilde f(s),\tilde \Omega(s)) = 0$ for $s \in
      (-\varepsilon,\varepsilon)$;
    \item $\tilde f(0) = 0$, $\tilde \Omega(0) = \Omega_{nm}$, $\tilde f_s(0) = 1/w^{nm-1}$, and $\tilde \Omega_s(0) = 0$; and
    \item \label{thm:locbif:uniq} there exists an open neighborhood $V$ of $(0,\Omega_{nm})$ in
      $X \by \R$ such that
      \begin{align*}
        V \cap (\soln \without \triv) = 
        \big\{(\tilde f(s),\tilde\Omega(s)) : 0 < \abs s < \varepsilon \big\}.
      \end{align*}
    \end{enumerate}
  \end{enumerate}
\end{theorem}
\begin{proof}
  We have already shown in Lemma~\ref{lem:analytic} that $\F$ is an
  analytic operator, and we have checked that $\F(0,\Omega)=0$ for all
  $\Omega$. Thus Lemma~\ref{lem:fredholm} in the next section implies
  that $\F_f(0,\Omega)$ is Fredholm with index 0 (this can also be
  verified directly). Expanding $g \in X$ in a Fourier series 
  \begin{align*}
    g(w)=\sum_{k\geq 1}\frac{a_k}{{w}^{km-1}},
  \end{align*}
  we see from Lemma~\ref{lem:Fdiff} that the operator $\F_f(0,\Omega)$
  in \eqref{eqn:Fdiff} is the Fourier multiplier
  \begin{equation}\label{df0}
    \F_f(0,\Omega) \sum_{k \ge 1}\frac{a_k}{w^{km-1}}
    =
    i\sum_{k\geq 1}km(\Omega-\Omega_{km}) a_k(\bar{w}^{km}-w^{km}). 
  \end{equation}
  If $\Omega$ is not one of the $\Omega_{nm}$, then
  \eqref{df0} shows that $\F_f(0,\Omega)$ has trivial kernel. Since it
  is Fredholm with index 0 it is therefore invertible, and 
  \ref{thm:locbif:imp} follows from the implicit function theorem.

  So consider the trivial solution $(0,\Omega_{nm})$ for some $n \ge
  1$. The strict monotonicity of the sequence $k\mapsto
  \Omega_{km}$ implies that the kernel of $\F_f(0,\Omega_{nm})$ is
  indeed the one-dimensional vector space spanned by $f_{nm}(w)=
  \bar{w}^{nm-1}$. Next we verify the transversality condition
  $\F_{f\Omega}(0,\Omega_{nm}) f_{nm} \notin \ran
  \F_f(0,\Omega_{nm})$. Differentiating \eqref{df0} with respect to
  $\Omega$ yields
  \begin{alignat*}{2} 
    \F_{f\Omega}(0,\Omega)g(w)&=2\F_f(0,\Omega)g(w)=
    {i}\sum_{k\geq 1}km a_{k}\big(\bar{w}^{km}-w^{km}\big),
  \end{alignat*} 
  and so we calculate
  \begin{align*}
    \F_{f\Omega}(0,\Omega_{nm})f_{nm}=i{nm} (\bar{w}^{nm}-w^{nm})\notin \ran(\F_{f}(0,\Omega_{nm}))
  \end{align*}
  as desired. Part \ref{thm:locbif:bif} of the theorem now follows
  immediately from Theorem~\ref{thm:genlocal}, except for the assertion
  that $\tilde \Omega_s(0) = 0$. 
 
  This derivative can be calculated directly using so-called
  ``bifurcation formulas'' \cite[Section~I.6]{kielhofer}. As is often
  the case for pitchfork bifurcations, however, the full calculation
  is unnecessary because of symmetry considerations. Set $n=1$ and let
  $\ell \in Y^*$ be a linear functional with $\ran
  \F_f(0,\tilde\Omega(0)) = \ker \ell$. Then the transversality
  condition becomes $\langle \ell,\F_{f\Omega}(0,\tilde\Omega(0))\tilde
  f_s(0)\rangle \ne 0$. Differentiating $\F(\tilde f(s),\tilde\Omega(s))= 0$
  with respect to $s$ we discover that 
  \begin{align*}
    \F_{ff}(0,\tilde\Omega(0)) [\tilde f_s(0),\tilde f_s(0)]
    + \F_f(0,\tilde\Omega(0)) \tilde f_{ss}(0)
    + 2\tilde\Omega_s(0)\F_{f\Omega}(0,\tilde\Omega(0)) \tilde f_s(0).
  \end{align*}
  Testing against $\ell$, the $\F_f$ term drops out and we are left
  with
  \begin{align*}
    \big\langle 
      \ell,
      \F_{ff}(0,\tilde\Omega(0)) [\tilde f_s(0),\tilde f_s(0)]
      + 2\tilde\Omega_s(0)\F_{f\Omega}(0,\tilde\Omega(0)) \tilde f_s(0)
    \big\rangle
    = 0
  \end{align*}
  and hence
  \begin{align}
    \label{eqn:avoidme}
    \tilde\Omega_s(0) = - \frac 12 
    \frac{\big\langle \ell,
    \F_{ff}(0,\tilde\Omega(0)) [\tilde f_s(0),\tilde f_s(0)]
    \big\rangle}
    {\big\langle \ell,\F_{f\Omega}(0,\tilde\Omega(0)) \tilde f_s(0)
    \big\rangle}.
  \end{align}
  
  Now we rotate by $\pi/m$ and repeat the above calculation. Consider
  the invertible linear maps $T_X \maps X \to X$ and $T_Y \maps Y \to
  Y$ defined by $T_Xf(w) = e^{-{i\pi/m}}f(e^{i\pi/m}w)$ and $T_Y h(w)
  = h(e^{i\pi/m}w)$. It is straightforward to verify that $\F$
  commutes with $T_X,T_Y$ in that $\F(T_X f,\Omega) = T_Y
  \F(f,\Omega)$. In particular, $\F(\tilde f(s),\tilde \Omega(s)) = 0$
  implies that $\F(T_X \tilde f(s),\tilde \Omega(s)) = 0$. The
  analogue of \eqref{eqn:avoidme} is then
  \begin{align*}
    \tilde\Omega_s(0) = - \frac 12 
    \frac{\big\langle \ell,
    \F_{ff}(0,\tilde\Omega(0)) [T_X \tilde f_s(0),T_X \tilde f_s(0)]
    \big\rangle}
    {\big\langle \ell,\F_{f\Omega}(0,\tilde\Omega(0)) T_X \tilde f_s(0)
    \big\rangle},
  \end{align*}
  Plugging in the fact that $T_X \tilde f_s(0) = - \tilde f_s(0)$, we
  obtain
  \begin{align*}
    \tilde\Omega_s(0) = + \frac 12 
    \frac{\big\langle \ell,
    \F_{ff}(0,\tilde\Omega(0)) [\tilde f_s(0),\tilde f_s(0)]
    \big\rangle}
    {\big\langle \ell,\F_{f\Omega}(0,\tilde\Omega(0)) \tilde f_s(0)
    \big\rangle},
  \end{align*}
  which differs from \eqref{eqn:avoidme} only in sign. Thus
  $\tilde\Omega_s(0) = 0$ as desired. 

  For $n > 1$ one can repeat the above calculation of $\tilde
  \Omega_s(0)$ by replacing $X,Y$ with the subspaces $X_n,Y_n$ with
  $nm$-fold symmetry and then use the uniqueness in
  \ref{thm:locbif:uniq}.
\end{proof}

The global curve $\cm$ which we will construct is a continuation
of the local curve constructed in Theorem~\ref{thm:locbif} that
bifurcates from $(0,\Omega_m)$. As hinted at in the above proof, the
portions of the curve with $s > 0$ and $s < 0$ are related by the
symmetry $T_X$, and so there is no loss of generality in restricting
to $s > 0$. Thus we make the following definition.
\begin{definition}[The local curve $\cm_\loc$]\label{def:loc}
  With $\varepsilon, \tilde f, \tilde \Omega$ as in
  Theorem~\ref{thm:locbif}\ref{thm:locbif:bif} with $n=1$, we define
  $\cm_\loc \sub \soln$ to be the portion of the bifurcation curve
  with $s > 0$, that is 
  \begin{align*}
    \cm_\loc := \{ (\tilde f(s),\tilde \Omega(s)) : 0 < s < \varepsilon \}.
  \end{align*}
\end{definition}

\section{Global bifurcation}\label{sec:global}

In this section we apply an abstract result on analytic global
bifurcation (Theorem~\ref{thm:genglobal}) to extend the local curve
$\cm_\loc$ from the previous section to a global one. In order to
apply Theorem~\ref{thm:genglobal}, we need to verify that the
linearized operators $\F_f(f,\Omega)$ that we will encounter along
this curve are Fredholm with index zero, and also that the curve has
certain compactness properties. Both of these tasks will be
accomplished by viewing $\F(f,\Omega)=0$ as the quasilinear
Riemann--Hilbert-type problem $\Im(A\phi')=0$, applying
Lemma~\ref{lem:rh}, and using the good control we have over the
coefficient $A$.

We first verify that the coefficient $A$ appearing in our
Riemann--Hilbert problem has winding number zero.
\begin{lemma}[Winding number]\label{lem:winding}
  Suppose that $(\phi-w,\Omega) \in \soln^{k+\beta}$ for some integer
  $k \ge 1$ and $\beta \in (0,1)$. Then $A := ( \Omega\bar{ \phi } +
  \tfrac 12 \Ca(\phi) \bar\phi) w$ has winding number $0$ in that
  \begin{align}
    \label{eqn:windingsense}
    A \ne 0,
    \qquad \arg A(e^{it})\Big|^{t=2\pi}_{t=0} = 0
  \end{align}
  for some continuous branch of $\arg$.
  \begin{proof}
    By \eqref{eqn:rh} we have $\Im(A\phi') = 0$, while
    $(\phi-w,\Omega) \in U_2^{k+\beta} \cap U_3^{k+\beta}$ implies $A
    \ne 0$ and $\phi' \ne 0$. Thus $A = \lambda \bar{\phi'}$ for some
    real-valued and non-vanishing $\lambda \in C^{k-1+\beta}(\T)$,
    and it suffices to show that $\phi'$ has winding number zero with
    respect to the origin in the sense of \eqref{eqn:windingsense}. As
    in Section~\ref{sec:space}, let $\Phi$ be the holomorphic
    extension of $\phi$ to $\C \without \D$. Then $\phi-w \in
    X^{k+\beta}$ implies $\Phi' = O(1)$ and $\Phi'' = O(1/w^{m+1})$ as
    $\abs w \to \infty$, and so in particular $\Phi''/\Phi' =
    O(1/w^3)$. Thus
    \begin{align*}
      \arg \phi'(e^{it})\Big|^{t=2\pi}_{t=0} 
      = \lim_{r \to \infty} \arg \Phi'(re^{it})\Big|^{t=2\pi}_{t=0}
      = \lim_{r \to \infty} \frac 1{2\pi i}\int_{\abs w = r} \frac{\Phi''(w)}{\Phi'(w)}\, dw = 0
    \end{align*}
    as desired.
  \end{proof}
\end{lemma}

With Lemma~\ref{lem:winding} in hand, we can now use
Lemma~\ref{lem:rh} to establish Fredholm properties for the linearized
operators $\F_f(f,\Omega)$.

\begin{lemma}[Fredholm index~0]\label{lem:fredholm}
  For any $(f,\Omega) \in \soln$, 
  the linearized operator $\F_f(f,\Omega) \maps X \to Y$ is Fredholm
  with index~0.
  \begin{proof}
    By Theorem~\ref{thm:cauchy} we can write
    \begin{align*}
      \F(f,\Omega) &= 
      \Im \left\{
        \big(
      \Omega(\bar w + \bar f)
      + \tfrac 12 \Ca(w + f) (\bar w + \bar f)
      \big)
      w (1 + f')
      \right\}\\
      &=: \Im \{ \mathscr A(f,\Omega) (1+f') \},
    \end{align*}
    where $\mathscr A$ is analytic
    $U_2^{k+\beta} \to C^{k+\beta}(\T)$ for any $k \ge 1$ and
    $\beta \in (0,1)$. Differentiating, we find 
    \begin{align*}
      % \label{eqn:secondcompact}
      \F_f(f,\Omega)g = 
      \Im\{\mathscr A(f,\Omega)g'\}
      + \Im\{(1+f') \mathscr A_f(f,\Omega)g\}
      =: L_1 g + L_2 g.
    \end{align*}
    By Lemmas~\ref{lem:analytic} and \ref{lem:winding}, $A = \mathscr
    A(f,\Omega)$ satisfies the hypotheses of Lemma~\ref{lem:rh} and so
    $L_1 \maps X \to Y$ is invertible. We claim that $L_2 \maps X \to
    Y$ is compact. Let $g_n$ be a bounded sequence in
    $X=X^{3+\alpha}$, and extract a subsequence so that $g_n \to g$ in
    $X^{3+\alpha/2}$. The analyticity of $\mathscr A \maps
    U^{3+\alpha/2} \to C^{3+\alpha/2}(\T)$ then guarantees that
    $\mathscr A_f(f,\Omega)g_n \to \mathscr A_f(f,\Omega)g_n$ in
    $C^{3+\alpha/2}(\T)$ and hence also in $Y=Y^{2+\alpha}$, proving
    the claim. Thus $\F_f(f,\Omega)$ is sum of an invertible operator
    and a compact operator and is hence Fredholm with index 0 as
    desired.
  \end{proof}
\end{lemma}

To prove the desired compactness properties for $\soln$, we introduce
a family of closed and bounded sets $E_\delta^{k+\beta} \sub
U^{k+\beta}$ which exhaust $U^{k+\beta}$ as $\delta \to 0$ and which
make the various conditions in the definitions of
$U_1^{k+\beta},U_2^{k+\beta},U_3^{k+\beta}$ quantitative. Here, as
always, the integer $k \ge 1$ and $\beta \in (0,1)$.

\begin{lemma}\label{lem:Edelta}
  For any $\delta > 0$, the set $E_\delta^{k+\beta} \sub X^{k+\beta}
  \by \R$ defined by the inequalities
  \begin{align}
    \label{eqn:delta}
    \min_\T 
    \abs{\Omega\bar{ \phi } + \tfrac 12 \Ca(\phi) \bar\phi}
    ,\,
    \frac 1{1+\abs\Omega+\n{ \phi}_{C^{k+\beta}}}
    ,\,
    \frac \pi 2 - \max_\T \left| \arg \frac{w \phi'}{\phi} \right|
    ,\,
    \min_\T \abs{\phi'}
    ,\,
    \min_\T \abs{\phi}  \ge \delta,
  \end{align}
  is a closed and bounded subset of $U^{k+\beta}$. Moreover, for any
  $(\phi-w,\Omega) \in U^{k+\beta}$ there exists $\delta > 0$ so that
  $(\phi-w,\Omega) \in E_\delta^{k+\beta}$.
  \begin{proof}
    First we prove the last statement. If $(\phi-w,\Omega) \in 
    U^{k+\beta}$, then the existence of a bound on $\abs {\Omega\bar{
      \phi } + \tfrac 12 \Ca(\phi) \bar\phi}$ follows from $(\phi-w,\Omega)
    \in  U_2^{k+\beta}$, the second bound is immediate, and the
    remaining bounds follow from $(\phi-w,\Omega) \in 
    U_3^{k+\beta}$.

    It remains to show that $E_\delta^{k+\beta} \sub U^{k+\beta}$ is
    closed and bounded. The boundedness is clear, as is the
    containment $E_\delta^{k+\beta} \sub  U_2^{k+\beta} \cap
    U_3^{k+\beta}$. The containment $E_\delta^{k+\beta} \sub
    U_1^{k+\beta}$ then follows from Lemma~\ref{lem:u1u3}. The second
    and final two conditions in \eqref{eqn:delta} are clearly closed,
    and the third is also closed in combination with the final two
    since they avoid the potential singularities in the $\arg$
    function. Finally, from the last three inequalities in
    \eqref{eqn:delta} one can check that $\dist_{ X^{k+\beta}\by
    \R}(E_\delta^{k+\beta},\dell  U_3^{k+\beta}) > 0$. Again by
    Lemma~\ref{lem:u1u3} we have $ U_3^{k+\beta} \sub U_1^{k+\beta}$,
    and so we conclude that the closure of $E^{k+\beta}_\delta$ is
    contained in $U_1^{k+\beta}$. Since the mapping $(\phi-w,\Omega)
    \mapsto \Omega\bar{ \phi } + \tfrac 12 \Ca(\phi) \bar\phi$ is
    analytic $ U_1^{k+\beta} \to C^{k+\beta}(\T)$ by
    Theorem~\ref{thm:cauchy}, the mapping $(\phi-w,\Omega) \mapsto
    \min_\T \abs{\Omega \bar\phi + \tfrac 12 \Ca(\phi)\bar\phi}$ is
    continuous, and so the first condition in \eqref{eqn:delta} is
    closed.
  \end{proof}
\end{lemma}

As the following lemma shows, solutions in $E_\delta^{1+\alpha}$ are
automatically $C^\infty$, with higher derivatives controlled by
$\delta$. The main ingredient is Lemma~\ref{lem:rh} on the solvability
of Riemann--Hilbert problems. We will prove the analyticity of $\dell
D$ in Section~\ref{sec:n:reg}.
\begin{lemma}[Local compactness of the solution set]\label{lem:extrareg}
  For any $\delta > 0$ and $k \ge 1$, the set $\soln^{1+\alpha} \cap
  E_\delta^{1+\alpha}$ is a compact subset of $C^k(\T) \by \R$. In
  particular, there is a constant $C(k,\delta) > 0$ so that any
  solution $(f,\Omega) \in E_\delta^{1+\alpha} \cap \soln^{1+\alpha}$ satisfies
  \begin{align}
    \label{eqn:uniform}
    \n f_{C^k(\T)} < C.
  \end{align}
  \begin{proof}
    Let $(f,\Omega) \in \soln^{k+\beta}$ for some $k \ge 1$ and $\beta
    \in (0,1)$. As in the proof of Lemma~\ref{lem:fredholm}, we set 
    \begin{align*}
      A = \mathscr A(f,\Omega) 
      = \big( \Omega(\bar w + \bar f) 
      + \tfrac 12 \Ca(w + f) (\bar w + \bar f) \big) w 
    \end{align*}
    and view $\F^{k+\beta}(f,\Omega) = 0$ as the Riemann--Hilbert
    problem $\Im(A\phi') = 0$. Applying Lemmas~\ref{lem:winding} and
    \ref{lem:rh}\ref{lem:rh:fund}, this Riemann--Hilbert problem can
    be explicitly solved to obtain
    \begin{align}
      \label{eqn:iteratethis}
      f'(w) &=
      \exp \left\{
        \frac {w}{2\pi} \int_\T
        \frac 1{\tau-w}
        \left[
          \frac 1\tau\arg\bigg( \frac{A(\tau)}{\bar{A(\tau)}}\bigg)
        -
        \frac 1w\arg\bigg( \frac{A(w)}{\bar{A(w)}}\bigg)
        \right]
        \, d\tau \right\} - 1
        =: \mathscr G(f,\Omega).
    \end{align}
    We first claim that the expression $\mathscr G(f,\Omega)$ on the
    right hand side of \eqref{eqn:iteratethis} defines a continuous
    mapping $\mathscr G \maps  \soln^{k+\beta} \to
    C^{k+\beta}(\T)$. Note that Theorem~\ref{thm:cauchy} guarantees
    that $\mathscr A$ is continuous $ U^{k+\beta} \to
    C^{k+\beta}(\T)$; the quotient $\mathscr A/\bar {\mathscr A}$ is
    continuous between the same spaces thanks to the restriction
    $\abs{\mathscr A}> 0$ embedded in the definition of $
    U^{k+\beta}$, and the same is then true for the argument
    $\arg(\mathscr A/\bar{\mathscr A})$ thanks to the fact that
    $\mathscr A/\bar{\mathscr A}$ has winding number 0 by
    Lemma~\ref{lem:winding}. The Cauchy integral operator appearing in
    \eqref{eqn:iteratethis} is a bounded linear operator from
    $C^{k+\beta}(\T)$ to itself; this is for instance a very special
    case of Theorem~\ref{thm:cauchy}. Composing with the exponential,
    we therefore obtain the desired continuity of $\mathscr G$ and the
    claim is proved.

    Setting $k=1$ and $\beta = \alpha$, we see that any $(f,\Omega) \in
    \soln^{1+\alpha}$ has $f' = \mathscr G(f,\Omega) \in
    C^{1+\alpha}(\T)$ and hence $f \in C^{2+\alpha}(\T)$. Moreover, the
    continuity of $\mathscr G \maps \soln^{1+\alpha} \to
    C^{1+\alpha}(\T)$ implies that the inclusion $\soln
    \hookrightarrow C^{2+\alpha}(\T) \by \R$ is continuous. 
    Iterating this argument with $k = 2$ and so on, we discover
    that $f \in C^{k+\alpha}(\T)$ for all $k$ and that the inclusions
    $\soln^{1+\alpha} \hookrightarrow C^{k+\alpha}(\T) \by \R$ are all
    continuous. 
   
    By Lemma~\ref{lem:Edelta},
    $E_\delta^{1+\alpha}$ is a closed and bounded subset of $
    U^{1+\alpha}$, and hence a compact subset of $ U^{1+\alpha/2}$.
    Since $ \soln^{1+\alpha/2} \sub U^{1+\alpha/2}$ is closed,
    $E_\delta^{1+\alpha} \cap \soln^{1+\alpha/2}$ is also a compact
    subset of $\soln^{1+\alpha/2}$. And since $\soln^{1+\alpha/2}$
    includes continuously into $C^k(\T) \by \R$, we conclude that
    $E_\delta^{1+\alpha} \cap \soln^{1+\alpha}$ is a compact  subset
    of $C^k(\T) \by \R$ as desired. In particular, it is bounded in
    $C^k(\T) \by \R$, which implies \eqref{eqn:uniform}.
  \end{proof}
\end{lemma}

We are now in a position to apply the following version of
Theorem~9.1.1 in \cite{bt:analytic} as modified in
\cite{csv:critical}. In the abstract setting of
Theorem~\ref{thm:genlocal}, let
\begin{align*}
  \cmg_\loc = \big\{(\tilde x(s),\tilde\lambda(s)) : 0 < s < \varepsilon \big\}
\end{align*}
be the portion of the local bifurcation curve with $s > 0$.
\begin{theorem}[Analytic global bifurcation]\label{thm:genglobal}
  In the setting of Theorem~\ref{thm:genlocal}, suppose in addition
  that
  \begin{enumerate}[label=\rm(\alph*),start=4]
  \item \label{thm:genglobal:fredholm} $\Fg_x(x,\lambda)$ is a Fredholm operator of index zero whenever
    $\Fg(x,\lambda) = 0$; and
  \item \label{thm:genglobal:compact} for some sequence $(Q_j)_{j=1}^\infty$ of bounded closed
    subsets of $U$ with $U = \cup_j Q_j$, the set $\{(x,\lambda) \in U :
    \Fg(x,\lambda) = 0\} \cap Q_j$ is compact for each $j$.
  \end{enumerate}
  Then there exists a continuous curve $\cmg$ of solutions, with
  \begin{align*}
    \cmg_\loc \sub \cmg = \{(\tilde x(s),\tilde \lambda(s)) : 0 < s <
    \infty \} \sub \Fg^{-1}(0)
  \end{align*}
  where $(\tilde x,\tilde \lambda) \maps (0,\infty) \to X \by \R$ is
  continuous, and such that one of the following occurs
  \begin{enumerate}[label=\rm(\roman*)]
  \item there exists $T > 0$ such that, after a reparametrization,
    $(\tilde x(s+T),\tilde \lambda (s+T)) = (\tilde x(s),\tilde
    \lambda (s))$;
  \item \label{thm:genglobal:blowup} for every $j \in \N$ there exists $s_j > 0$ such that
    $(\tilde x(s),\tilde \lambda(s)) \notin Q_j$ for $s > s_j$.
  \end{enumerate}
  Moreover, the curve $\cmg$ has a real-analytic parametrization
  locally around each of its points, and is unique (up to
  reparametrization).
\end{theorem}
Applying Theorem~\ref{thm:genglobal} to our problem we obtain the following
intermediate result.
\begin{theorem}[Global bifurcation of vortex patches]\label{thm:global}
  There exists a continuous curve $\cm$ of solutions which extends
  $\cm_\loc$,
  \begin{align*}
    \cm_\loc \sub \cm 
    = \{(\tilde f(s),\tilde \Omega(s)) : 0 < s < \infty \}
    \sub \soln
  \end{align*}
  and such that either
  \begin{enumerate}[label=\rm(\roman*)]
  \item\label{thm:global:loop} there exists $T > 0$ such that, after a reparametrization,
    $(\tilde f(s+T),\tilde \Omega(s+T)) = (\tilde f(s),\tilde
    \Omega(s))$; or
  \item\label{thm:global:blowup} as $s \to \infty$, 
    \begin{align}
      \label{eqn:global:blowup}
      \min\left\{
        \min_\T 
      \abs[\big]{\Omega\bar{ \tilde\phi } + \tfrac 12 \Ca(\phi) \bar{\tilde\phi}}
        ,\,
      \frac 1{1+\abs\Omega+\n{\tilde \phi}_{C^{1+\alpha}}},\,
      \frac \pi 2 - \max_\T \left| \arg \frac{w\tilde \phi'}{\tilde\phi} \right|,\,
      \min_\T \abs{\tilde\phi'},\,
      \min_\T \abs{\tilde\phi} \right\} \longrightarrow 0,
    \end{align}
    where here $\tilde \phi(s) = w + \tilde f(s)$.
  \end{enumerate}
  Moreover, the curve $\cm$ has a real-analytic parametrization
  locally around each of its points, and is unique (up to
  reparametrization).
  \begin{proof}
    We apply Theorem~\ref{thm:genglobal}. By Theorem~\ref{thm:locbif},
    $\F \maps U \to Y$ satisfies the hypotheses of
    Theorem~\ref{thm:genlocal}, and by Lemma~\ref{lem:fredholm} $\F$
    satisfies the Fredholm index zero
    assumption~\ref{thm:genglobal:fredholm}. Setting $Q_j =
    E_{1/j}^{3+\alpha}$, Lemmas~\ref{lem:Edelta} and
    \ref{lem:extrareg} guarantee that the remaining assumption
    \ref{thm:genglobal:compact} is satisfied. With this choice of
    $Q_j$, alternative \ref{thm:global:blowup} above is a restatement
    of alternative \ref{thm:genglobal:blowup} in
    Theorem~\ref{thm:genglobal}, except that the H\"older exponent has
    been reduced from $3+\alpha$ to $1+\alpha$ thanks to
    Lemma~\ref{lem:extrareg}.
  \end{proof}
\end{theorem}

\section{Streamlines}\label{sec:nodal}

In this section we study the level sets of the relative stream
function $\Psi$, including of course the boundary $\dell D$ of the
vortex patch. These curves, called streamlines, represent fluid
particle trajectories in the rotating frame. In addition to their
independent interest, the results in this section will allow us to
eliminate the alternative \ref{thm:global:loop} in
Theorem~\ref{thm:global} that the global curve $\cm$ of vortex patches
forms a closed loop. 

For a solution $(\phi-w,\Omega) \in \soln$, the
symmetries $\phi(e^{2\pi i/m}w) = e^{2\pi i/m}\phi(w)$ and $\phi(\bar
w) = \bar{\phi(w)}$ imply that the solution $\Psi$ of
\eqref{eqn:Psi:lap}--\eqref{eqn:Psi:reg} is $2\pi/m$-periodic and even
in the polar coordinate $\theta$. Thus it is enough to describe $\Psi$
on the fundamental sector $S$ with left and right boundary portions
$L,R$ given in polar coordinates by
\begin{align}
  \label{eqn:SLR}
  S = \big\{ re^{i\theta} : r > 0,\, 0 < \theta < \pi/m \big\},
  \quad 
  L = \big\{  r > 0,\, \theta = \pi/m \big\},
  \quad 
  R = \big\{  r > 0,\, \theta = 0 \big\}.
\end{align}
See Figure~\ref{fig:phase}a for an illustration.
\begin{figure}
  \centering
  \includegraphics[scale=1.1]{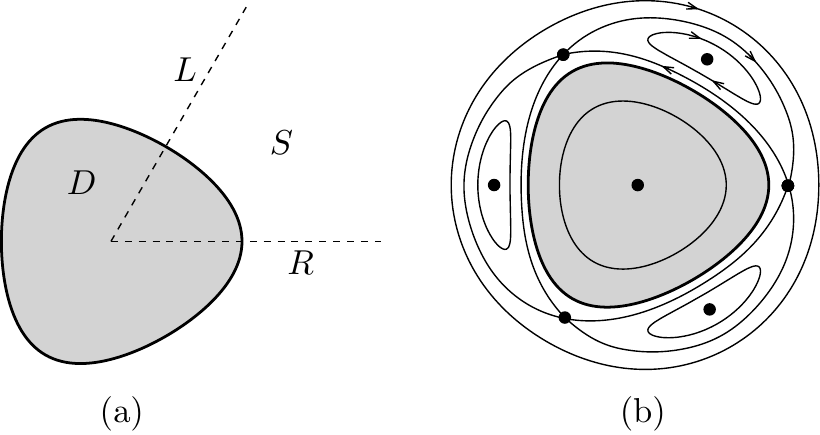}  
  \caption{(a) The sector $S$ and rays $L,R$ in \eqref{eqn:SLR}. (b)
  Level curves and critical points of $\Psi$.}
  \label{fig:phase}
\end{figure}
We will show in Section~\ref{sec:nodal:glob} that every
vortex patch in $\cm$ satisfies
\begin{subequations}\label{eqn:n:glob}
  \begin{alignat}{2}
    \label{eqn:n:r}
    \Psi_r              & > 0 && \ona \dell D, \\ 
    \label{eqn:n:t}
    \Psi_\theta         & > 0 && \ona S,   \\ 
    \label{eqn:n:tt:r}
    \Psi_{\theta\theta} & > 0 && \ona R,   \\ 
    \label{eqn:n:tt:l}
    \Psi_{\theta\theta} & < 0 && \ona L.
  \end{alignat}
\end{subequations}
With $r=\tilde r(\theta)$ a polar parametrization of the boundary $\dell D$
of the vortex patch, \eqref{eqn:n:glob} and the identity
$\Psi(\tilde r(\theta),\theta) = 0$ together imply the inequalities
\begin{align*}
    \tilde r'(\theta) < 0 \fora 0 < \theta < \frac \pi m,
    \quad 
    \tilde r''(0) < 0,
    \quad 
    \tilde r''\Big(\frac \pi m\Big) > 0
\end{align*}
claimed in Theorem~\ref{thm:informal}\ref{thm:informal:nodal}.

\begin{definition}[Nodal set]%\label{def:nodal}
  We define the ``nodal set'' $\nodal$ to be the subset of $\soln
  \without \triv$ where \eqref{eqn:n:glob} holds. 
\end{definition}

In Section~\ref{sec:nodal:loc}, we will show that the vortex patches
along the local curve $\cm_\loc$ not only satisfy \eqref{eqn:n:glob}
but also satisfy the additional inequalities
\begin{subequations}\label{eqn:n:loc}
  \begin{alignat}{2}
    \label{eqn:n:rt}
    \Psi^-_{r\theta}      & < 0 && \ona S \without D,   \\ 
    \label{eqn:n:rr}
    (r\dell_r)^2 \Psi^-   & < 0 && \ona \C \without D,   \\ 
    \label{eqn:n:rtt:r}
    \Psi^-_{r\theta\theta}& < 0 && \ona R \without D, \\
    \label{eqn:n:rtt:l}
    \Psi^-_{r\theta\theta}& > 0 && \ona L \without D,
  \end{alignat}
\end{subequations}
which imply that the contour plot of $\Psi$ on $\C \without D$ looks
qualitatively like Figure~\ref{fig:phase}b; see
Theorem~\ref{thm:phase}.

\subsection{Preliminary lemmas}%\label{sec:nodal:lem}
First we prove two simple lemmas which will be useful for both the
local and global arguments to follow. 
\begin{lemma}[Robustness of simple roots]\label{lem:elemcont}
  Consider the Banach space
  \begin{align*}
    Z = \big\{ g \in C^1([0,1],\R) : g(0) = 0\big\} 
  \end{align*}
  of $C^1$ functions vanishing at $0$, and the subset
  \begin{align*}
    V = \big\{ g \in Z : 
    g(t) > 0 \fora t > 0
    ,\,
    g'(0) > 0\big\}
  \end{align*}
  of functions which are positive away from $0$ and have a simple
  root. Then
  \begin{enumerate}[label=\rm(\alph*)]
  \item \label{lem:elemcont:open} $V \sub Z$ is open.
  \item \label{lem:elemcont:diff} If $G \maps (-1,1) \to Z$ is
    a $C^1$ map with $G(0) = 0$ and $G'(0) \in V$, then $G(s) \in V$
    for $s > 0$ sufficiently small.
  \end{enumerate}
  \begin{proof}
    First we prove \ref{lem:elemcont:open}. Let $g_0 \in V$ and let $g
    \in Z$ have $\n{g-g_0}_{C^1} < \varepsilon$ for some $\varepsilon
    > 0$ to be determined. From $g_0 \in V$ we deduce that $\dell_t
    g_0 > 0 \ona [0,\delta]$ and $g_0 > 0 \ona [\delta,1]$ for some
    $\delta > 0$. As the above intervals are compact, we can choose
    $\varepsilon$ small enough that the same strict inequalities hold
    for $g$. The remaining inequality $g(t) > 0$ for $t \in
    (0,\delta)$ then follows from the mean value theorem.
    From the differentiability of $G$ we have $\n{s^{-1}G(s) -
    G'(0)}_Z \to 0$ as $s \searrow 0$, and so $s^{-1} G(s) \in V$ for
    $s$ sufficiently small. Since $V$ is invariant under
    multiplication with positive scalars, this proves
    \ref{lem:elemcont:diff}.
  \end{proof}
\end{lemma}

\begin{lemma}[Weaker streamline conditions]\label{lem:nodweak}
  For a solution in $\soln \without \triv$, the conditions
  \eqref{eqn:n:glob} are equivalent to the weaker conditions
  \begin{subequations}\label{eqn:nw}
    \begin{alignat}{2}
      \label{eqn:nw:r}
      \Psi_r              & > 0 && \ona \dell D, \\ 
      \label{eqn:nw:t}
      \Psi_\theta         & > 0 && \ona \dell D,   \\ 
      \label{eqn:nw:tt:r}
      \Psi_{\theta\theta} & > 0 && \ona R \cap \dell D,   \\ 
      \label{eqn:nw:tt:l}
      \Psi_{\theta\theta} & < 0 && \ona L \cap \dell D.
    \end{alignat}
  \end{subequations}
  \begin{proof}
    Suppose that \eqref{eqn:nw} holds. Then \eqref{eqn:n:r} since it
    is just \eqref{eqn:nw:r}, and so it suffices to prove
    \eqref{eqn:n:t}--\eqref{eqn:n:tt:l}.
    To show \eqref{eqn:n:t}, we use the maximum principle. Differentiating \eqref{eqn:Psi:lap} shows that $\Psi_\theta$ is
    harmonic on both $S \cap D$ and $S \without D$, is continuous
    across $\dell D$, and vanishes at infinity. By symmetry, we also
    have $\Psi_\theta = 0$ on $\dell S$. Thus if $\inf_S \Psi_\theta <
    0$, by the maximum principle this infimum would have to be
    achieved on $S \cap \dell D$, a contradiction. Therefore $\inf_S
    \Psi_\theta = 0$, and so the strong maximum principle implies
    $\Psi_\theta > 0$ on $S$. Applying the Hopf lemma on $R$ and $L$,
    we see that \eqref{eqn:n:tt:r} and \eqref{eqn:n:tt:l} hold
    except potentially on $R \cap \dell D$ and $L \cap \dell D$. But
    there the inequalities follow from \eqref{eqn:nw:tt:r} and
    \eqref{eqn:nw:tt:l}, and so \eqref{eqn:n:glob} holds as desired.
  \end{proof}
\end{lemma}

\subsection{Streamlines of small solutions}\label{sec:nodal:loc}

In this section we prove that both \eqref{eqn:n:glob} and
\eqref{eqn:n:loc} hold along $\cm_\loc$. 

We begin by using Lemma~\ref{lem:stream} to 
translate the expansion for the conformal
mapping $\phi$ in Theorem~\ref{thm:locbif}\ref{thm:locbif:bif} into
expansions for the derivatives of $\Psi$ restricted $\dell D$.

\begin{lemma}[Expansions for derivatives of $\Psi$]\label{lem:Psir}
  Along the local curve $\cm_\loc$ of vortex patches $(\tilde
  f(s),\tilde \Omega(s))$, the partial derivatives of the
  corresponding relative stream functions
  $\Psi^-=\Psi|_{\C \without D}$ have the
  following expansions:
  \begin{align}
    \label{eqn:Psirexp}
    \begin{aligned}
      \dell_\theta \Psi^-(\phi(w);s)
      &= s \tfrac 1 2 \Im w^m + O(s^2),
      \\
      r\dell_r \Psi^- (\phi(w);s)
      & = \tfrac 1{2m}  + O(s),
      \\
      \dell_\theta r\dell_r \Psi^-(\phi(w);s)
      &= -s \tfrac m2 \Im w^m + O(s^2),
      \\
      (r\dell_r )^2\Psi^- (\phi(w);s)
      & = -\tfrac{m-1}m  
      + O(s),\\
      \dell_\theta^2 \Psi^- (\phi(w);s)
      &= s \tfrac m 2 \Re w^m + O(s^2),
      \\
      \dell_\theta^2 r\dell_r \Psi^-(\phi(w);s)
      &= -s \tfrac{m^2}2 \Re w^m + O(s^2).
    \end{aligned}
  \end{align}
  \begin{proof}
    By Theorem~\ref{thm:locbif}\ref{thm:locbif:bif},
    we have the asymptotic expansions
    \begin{align}
      \label{eqn:phis}
      \phi(w;s) := w + \tilde f(s)(w) = 
      w + \frac s{w^{m-1}} + O(s^2),
      \qquad 
      \tilde\Omega(s) = \frac{m-1}{2m} + O(s^2)
    \end{align}
    in $C^{3+\alpha}(\T)$ and $\R$ respectively. From
    Lemma~\ref{lem:stream} and the analyticity of the Cauchy operator,
    we know that the compositions $\dell_z^k\Psi^- \circ \phi$ depend
    analytically on $s$ as elements of $C^{4-k+\alpha}(\T)$ for
    $k=1,2,3$. Inserting the expansion \eqref{eqn:phis} into
    \eqref{eqn:Psizw} and repeatedly using the calculus of residues as
    in the proof of Lemma~\ref{lem:Fdiff}, we eventually find
    \begin{align}
      \label{eqn:Psizwexp}
      \begin{aligned}
        \dell_z\Psi^-(\phi(w;s);s) 
        &= \frac 1{4 w} - \frac {m-1}{4m} \bar w
        - s\frac{m-1}{2m} \left( \frac{\bar w}{w^{m-1}} 
        + \frac w{\bar w^{m-1}} \right)
        + O(s^2),
        \\
        \dell_z^2 \Psi^-(\phi(w;s);s) 
        &= -\frac 1{4w^2}
        - s\frac{m-1}2 \frac 1{w^{m+2}} + O(s^2),\\
        \dell^3_z \Psi^-(\phi(w;s);s) 
        &= \frac 12 \frac 1{w^3}
        + s \frac{(m-1)(m+4)}4 \frac 1{w^{m+3}} + O(s^2). 
      \end{aligned}
    \end{align}
    Using the convenient formulas $r\dell_r = z\dell_z + \bar z
    \dell_{\bar z}$ and $\dell_\theta = i(z\dell_z - \bar z
    \dell_{\bar z})$, the partials of $\Psi^-$ can then
    be easily calculated in terms of
    $\dell_z\Psi^-,\dell^2_z\Psi^-,\dell^3_z\Psi^-$:
    \begin{align}
      \label{eqn:Psiz}
      \begin{aligned}
        \dell_\theta \Psi^-(z) 
        &= -2\Im[z\dell_z \Psi^-],
        \\
        r\dell_r \Psi^-(z) 
        &= 
        2 \Re \left[ z\dell_z \Psi^- \right] ,
        \\
        \dell_\theta r\dell_r \Psi^-(z)
        &= 
        -2 \Im \left[ z\dell_z \Psi^- + z^2 \dell_z^2 \Psi^- \right],
        \\
        (r\dell_r )^2\Psi^-(z) 
        &= 
        2\Re \left[ z \dell_z \Psi^- + z^2 \dell_z^2 \Psi^- \right] 
        - \Omega \abs z^2,\\
        \dell_\theta^2 \Psi^-(z) 
        &= -2\Re \left[ z\dell_z\Psi^- + z^2\dell_z^2 \Psi^- \right]
        - \Omega \abs z^2,
        \\
        \dell_\theta^2 r\dell_r \Psi^-(z)
        &=
        -2\Re \left[z\dell_z \Psi^- 
        + 3z^2 \dell_z^2 \Psi^-
        + z^3 \dell_z^3 \Psi^-\right]
        - \Omega \abs z^2,
      \end{aligned}
    \end{align}
  where here we repeatedly have used the identity $\dell_z \dell_{\bar
  z} \Psi^- = \dell_{\bar z} \dell_z \Psi^- = \tfrac 14 \Delta \Psi^- =
  -\tfrac\Omega 2$.
  Plugging \eqref{eqn:Psizwexp} into \eqref{eqn:Psiz}, and using the
  identity $\bar w = 1/w$ for $w \in \T$, we finally obtain
  \eqref{eqn:Psirexp} as desired.
\end{proof}
\end{lemma}

We can now establish \eqref{eqn:n:glob} and \eqref{eqn:n:loc} by using
the expansions in Lemma~\ref{lem:Psir},  Lemma~\ref{lem:elemcont}, and
several non-obvious maximum principle arguments.

\begin{proposition}[Signs of derivatives of $\Psi$]\label{prop:Psisign}
  The inequalities
  \eqref{eqn:n:glob} and \eqref{eqn:n:loc} hold along $\cm_\loc$,
  after possibly reducing $\varepsilon$ in Definition~\ref{def:loc}.
  \begin{proof}
    By Lemma~\ref{lem:nodweak}, to prove \eqref{eqn:n:glob} it
    suffices to prove \eqref{eqn:nw}. The first inequality
    \eqref{eqn:nw:r} as well as the last two
    \eqref{eqn:nw:tt:r}--\eqref{eqn:nw:tt:l} follow
    immediately from the expansion \eqref{eqn:Psirexp} in
    Lemma~\ref{lem:Psir}.

    Thus in order to prove \eqref{eqn:n:glob} it remains only to show that
    \eqref{eqn:nw:t} holds. For this, we use the polar coordinate
    representation
    \begin{align*}
      \phi(e^{it}) = \rho(t) e^{i\vartheta(t)}
    \end{align*}
    introduced in Lemma~\ref{lem:polar}, where we are temporarily
    suppressing dependence on $s$.  The symmetries \eqref{eqn:sym}
    guarantee that $\rho$ is even and $2\pi/m$-periodic in $t$, and
    hence in particular that $\rho'(t) = 0$ for $t=0,\pi/m$.
    Therefore, if $g \in C^1(\C \without D)$, we have 
    \begin{align*}
      \frac d{dt} g(\phi(e^{it}))
      &= \dell_r g(\phi(e^{it})) \rho'(t)
      + \dell_\theta g(\phi(e^{it})) \vartheta'(t)\\
      &= \Re \left( \frac{e^{it}\phi'(e^{it})}{\phi(e^{it})} \right)
      \dell_\theta g(\phi(e^{it}))
      \quad \ata t = 0,\pi/m.
    \end{align*}
    Setting $w=e^{it}$, we abbreviate this as
    \begin{align}
      \label{eqn:abcorn}
      \frac d{dt} = \Re \left( \frac{w\phi'}{\phi} \right)
      \frac \dell{\dell\theta}
      \qquad \ata t = 0,\pi/m.
    \end{align}
    
    Reintroducing the dependence on $s$, \eqref{eqn:nw:t}
    is equivalent to 
    \begin{align}
      \label{eqn:desiredg}
      g(t;s) := \dell_\theta\Psi(\phi(e^{it};s);s) > 0 
      \quad \fora 0 < t < \pi/m.
    \end{align}
    The symmetries \eqref{eqn:sym} guarantee that $g$ is an odd and
    $2\pi/m$-periodic function of $t$. Moreover, as in the proof of
    Lemma~\ref{lem:Psir}, $g$ depends analytically on $s$ as a
    $C^{3+\alpha}$ function, and the expansion \eqref{eqn:Psirexp}
    gives $g(t;0) \equiv 0$ and
    \begin{align*}
      g_s(t;0) = \tfrac 12 \sin mt  > 0 
      \quad 
      \fora 0 < t < \pi/m.
    \end{align*}
    Finally, at $t = 0$ the expansions \eqref{eqn:phis} and
    \eqref{eqn:Psirexp} yield
    \begin{align*}
      g_t(0;s) 
      &= \Re \left( \frac{\phi'}{\phi} \right)
      \dell_{\theta\theta} \Psi(\phi(1;s);s) 
      = s \tfrac m2 + O(s^2),
    \end{align*}
    while at $t=\pi/m$ we similarly obtain
    \begin{align*}
      g_t(\pi/m;s) = -s\tfrac m2 + O(s^2).
    \end{align*}
    Applying Lemma~\ref{lem:elemcont}\ref{lem:elemcont:diff}, we
    deduce that \eqref{eqn:desiredg} and hence \eqref{eqn:nw:t}
    hold for $s > 0$ sufficiently small.

    The proof of \eqref{eqn:n:rt}, \eqref{eqn:n:rtt:r}, and
    \eqref{eqn:n:rtt:l} is quite similar. Setting 
    \begin{align*}
      g(t;s) = \dell_\theta r \dell_r\Psi^-(\phi(e^{it};s);s),
    \end{align*}
    $g$ is again an odd and $2\pi/m$-periodic function of $t$, this
    time with the expansions
    \begin{align*}
      g(t;s) = -s \tfrac m2 \sin mt + O(s^2),
      \qquad 
      g_t(0;s) = -s \tfrac{m^2}2 + O(s^2),
      \qquad 
      g_t(\pi/m;s) = +s \tfrac{m^2}2 + O(s^2).
    \end{align*}
    Lemma~\ref{lem:elemcont}\ref{lem:elemcont:diff} then guarantees
    that $g < 0$ for $0 < t < \pi/m$ and $s > 0$ sufficiently small,
    and hence that $\dell_\theta r\dell_r \Psi < 0$ on $S \cap \dell
    D$. Now $\dell_\theta r\dell_r \Psi$ is harmonic on $S \without
    D$, vanishes at infinity, and vanishes along $\dell S$. Thus, as
    in the proof of Lemma~\ref{lem:nodweak}, the maximum principle
    forces $\dell_\theta r \dell_r \Psi < 0$ on $S \without D$. The
    inequalities \eqref{eqn:n:rtt:r} and \eqref{eqn:n:rtt:l}
    hold on $R \cap \dell D$ and $L \cap \dell D$ by the expansion
    \eqref{eqn:Psirexp}, and on $R \without \dell D$ and $L \without
    \dell D$ by the Hopf lemma applied to $\dell_\theta r \dell_r
    \Psi$.

    The only remaining inequality is \eqref{eqn:n:rr}. A direct
    calculation shows that the function
    \begin{align*}
      \varphi(r,\theta) = (r\dell_r)^2 \Psi^-(r,\theta) + 2\Omega r^2
    \end{align*}
    is harmonic. As with our other harmonic functions, it vanishes at
    infinity. From the expansions \eqref{eqn:Psirexp} and
    \eqref{eqn:phis}, $\varphi(\phi(w;s);s) = O(s)$, and so by the
    maximum principle
    \begin{align*}
      \n\varphi_{L^\infty(\C \without D)} = O(s).
    \end{align*}
    We also of course have
    \begin{align*}
      \min_{z \in \dell D(s)} \Omega(s)\abs z^2
      = \min_{w\in \T} \Omega(s)\abs{\phi(w;s)}^2 = \frac{m-1}{2m} - O(s).
    \end{align*}
    Combining these two facts, we find that on $\C \without D$,
    \begin{align*}
      (r\dell_r)^2\Psi(z;s) 
      &= \varphi(z;s) - 2\Omega \abs z^2\\
      &< \n\varphi_{L^\infty(\C \without D)} -
        2\min_{z \in \dell D} \Omega(s)\abs z^2\\
        &= -\tfrac{m-1}{m} + O(s) < 0
    \end{align*}
    for $s$ sufficiently small.
\end{proof}
\end{proposition}

Combining Proposition~\ref{prop:Psisign} with simpler arguments
for the other small solutions described by Theorem~\ref{thm:locbif},
we arrive at the following characterization of $\nodal$ near $\triv$.
\begin{lemma}[Nodal properties of small solutions]\label{lem:nodloc}
  \hfill
  \begin{enumerate}[label=\rm(\alph*)]
  \item If $\Omega \ne \Omega_m$, then there exists a
    neighborhood $V \sub U$ of $(0,\Omega)$ such that
    $V \cap \nodal = \varnothing$.
  \item \label{lem:nodloc:bif} 
    There exists a neighborhood $V \sub U$ of
    $(0,\Omega_m)$ such that $V \cap \nodal = \cm_\loc$.
  \end{enumerate}
  \begin{proof}
    We apply Theorem~\ref{thm:locbif}. If $\Omega$ is not one of the
    $\Omega_{nm}$, then we can simply choose $V$ by
    part~\ref{thm:locbif:imp} of that theorem to get $V \cap \nodal
    \sub \nodal \cap \triv = \varnothing$. If, on the other hand,
    $\Omega = \Omega_{nm}$ for some $n > 1$, then by
    part~\ref{thm:locbif:bif} of Theorem~\ref{thm:locbif}, we know
    that $V$ can be chosen so that $V \cap (\soln \without \triv)$
    consists of solutions 
    \begin{align*}
      f = \tilde f(s) = \frac s{w^{nm-1}} + o(s),
    \end{align*}
    where $s \ne 0$ is the parameter along the bifurcation curve, and the
    remainder refers to a term which is $o(s)$ in $X$. Arguing as in
    Lemma~\ref{lem:Psir}, we discover that 
    \begin{align*}
      \dell_\theta\Psi (\phi(w;s);s)
      = -s \tfrac 12 \Im w^{-nm} + O(s^2).
    \end{align*}
    In particular, $\dell_\theta \Psi$ changes sign on $S \cap \dell
    D$ for $s \ne 0$ sufficiently small, contradicting \eqref{eqn:nw:t}.
    By Proposition~\ref{prop:Psisign}, solutions with $n=1$ and $s >
    0$ is sufficiently small, i.e.~solutions on $\cm_\loc$, do lie in
    $\nodal$. On the other hand for $s < 0$ small we have 
    \begin{align*}
      \dell_\theta\Psi (\phi(e^{\pi i/2m};s);s)
      = -s \tfrac 12  + O(s^2) < 0,
    \end{align*}
    again contradicting \eqref{eqn:nw:t}.
  \end{proof}
\end{lemma}

With Proposition~\ref{prop:Psisign} in hand, we can also complete the
proof of Theorem~\ref{thm:phase}.
\begin{proof}[Proof of Theorem~\ref{thm:phase}]
  First we claim that there exists a unique polar curve
  $r=r_c(\theta)$ in $\overline S \without D$ such that $\Psi_r > 0$
  for $r < r_c$ and $\Psi_r < 0$ for $r > r_c$. Since $\Psi_r > 0$ on
  $\dell D$ by \eqref{eqn:n:r} and $\Psi_r \to -\infty$ as $r \to
  \infty$ by \eqref{eqn:Psi:asym}, for each $\theta$ there exists a
  radius $r_c(\theta)$ such that $(r,\theta) \notin D$ and
  $\Psi_r(r_c(\theta),\theta) = 0$. From \eqref{eqn:n:rr} we see that
  $\Psi_{rr} = (r\dell r)^2\Psi/r^2 < 0$ along $r=r_c$, which gives
  the uniqueness and also $r_c \in C^{2+\alpha}$. Indeed, by standard
  elliptic theory $\Psi$ is analytic away from $\dell D$, and so by
  the analytic implicit function theorem $r_c$ is analytic. We
  note that implicit differentiation yields $r_c'  =
  -\Psi_{r\theta}/\Psi_{rr} < 0$.

  From \eqref{eqn:n:t} and the above arguments, the unique
  equilibria in $\overline S \without D$ are $P := (r_c(0),0)$ and $Q :=
  (r_c(\pi/m),\pi/m)$. At $P$, \eqref{eqn:n:rr} and
  \eqref{eqn:n:tt:r} give
  \begin{align*}
    \Psi_{\theta\theta}\Psi_{rr} - 2\Psi_{\theta r}^2
    = \Psi_{\theta\theta} \Psi_{rr}
    < 0
  \end{align*}
  so that $P$ is a saddle. Similarly \eqref{eqn:n:rr} and
  \eqref{eqn:n:tt:l} guarantee that $Q$ is a center.
  The remaining statements now follow from a straightforward nullcline
  analysis (in polar coordinates).
\end{proof}

\subsection{Streamlines for the global curve}
\label{sec:nodal:glob}

We now show that the inequalities \eqref{eqn:n:glob} hold not just
along the local curve $\cm_\loc$ but also along the global curve $\cm$. This
will be sufficient to eliminate the alternative \ref{thm:global:loop}
in Theorem~\ref{thm:global} that $\cm$ forms a closed loop. Unlike in 
Section~\ref{sec:nodal:loc}, we will not be able to rely on asymptotic
expansions, and will instead have to use more subtle arguments
involving the structure of the equations.
Conjecture~\ref{conj:nodal} is that the inequalities \eqref{eqn:n:loc}
also hold along $\cm$, but we only have numerical evidence of this
fact; see Section~\ref{sec:numerical}.

\begin{proposition}[Robustness of the nodal set]\label{prop:openclosed}
  The nodal set $\nodal$ is both relatively open and relatively closed
  in $\soln \without \triv$.
  \begin{proof}
    First we claim that $\nodal$ is relatively open. To be concrete,
    fix $(f^0,\Omega^0) \in \nodal$ and consider $(f^1,\Omega^1) \in
    \soln \without \triv$ with $\n{(f^1,\Omega^1)-(f^0,\Omega^0)}_{X
    \by \R} < \varepsilon$ where $\varepsilon > 0$ is to be
    determined. Let $\Psi^0,\Psi^1$ be the corresponding relative
    stream functions. Then $\Psi^0$ satisfies \eqref{eqn:n:glob}, and
    by Lemma~\ref{lem:nodweak} we will have $(f^1,\Omega^1) \in
    \nodal$ as soon as $\Psi^1$ satisfies \eqref{eqn:nw}. From
    Lemma~\ref{lem:stream} and \eqref{eqn:Psiz}, we know that the
    composition $\dell_r\Psi \circ \phi$ depends continuously on
    $(f,\Omega) \in \soln$ as an element of $C^{3+\alpha}(\T)$. In
    particular, we can choose $\varepsilon > 0$ small enough that
    $\Psi^1$ satisfies \eqref{eqn:nw:r}. For the remaining
    inequalities, we consider the composition
    \begin{align*}
      g(t) := \dell_\theta\Psi(\phi(e^{it}))
    \end{align*}
    so that \eqref{eqn:nw:t} is equivalent to $g > 0$ for $0 < t <
    \pi/m$.
    Again thanks to Lemma~\ref{lem:stream} and \eqref{eqn:Psiz}, we
    see that $g$ depends continuously on $(f,\Omega) \in \soln \without
    \triv$ as an element of $C^{3+\alpha}[0,2\pi]$. Moreover, by
    symmetry, $g$ is an odd and $2\pi/m$-periodic function of $t$.
    Differentiating using \eqref{eqn:abcorn} yields
    \begin{align*}
      g'(t) 
      = \Re \left( \frac{e^{it}\phi'(e^{it})}{\phi(e^{it})} \right)
      \dell_{\theta\theta} \Psi(\phi(e^{it})) 
      \qquad \ata t=0,\pi/m.
    \end{align*}
    For $(f,\Omega) \in \soln$, the factor multiplying
    $\dell_{\theta\theta}\Psi$ above is strictly positive, and indeed
    we can choose $\varepsilon > 0$ so that there is a uniform lower
    bound for the $(f^1,\Omega^1)$ under consideration. Thus, in this
    neighborhood, $g'(0) > 0$ is equivalent to \eqref{eqn:nw:tt:l} and
    $g'(\pi/m) < 0$ is equivalent to \eqref{eqn:nw:tt:r}. Applying
    Lemma~\ref{lem:elemcont}\ref{lem:elemcont:open}, we conclude that
    there exists $\delta > 0$ so that $\n{g^1-g^0}_{C^1} < \delta$
    implies that $g^1 > 0$ for $0 < t < \pi/m$, $(g^1)'(0) > 0$, and
    $(g^1)'(\pi/m) < 0$. Thanks to the continuous dependence of the
    compositions $\Psi_\theta \circ \phi$, $\Psi_{\theta\theta} \circ
    \phi$ on $(f,\Omega) \in \soln$ as elements of $C^{2+\alpha}(\T)$,
    we can therefore choose $\varepsilon > 0$ in terms of $\delta$ so
    that $\Psi^1$ satisfies \eqref{eqn:nw:t}--\eqref{eqn:nw:tt:l} as
    desired.
    
    We now come to the core of the proof: showing that $\nodal$, which
    is defined in terms of strict inequalities, is nevertheless
    relatively closed. Suppose that a sequence of patches
    $(f^n,\Omega^n) \in \nodal$ converges to some patch $(f,\Omega)
    \in \soln \without \triv$, and let $\Psi^n,\Psi$ be the associated
    relative stream functions. Taking limits in \eqref{eqn:nw}, we
    see that 
    \begin{subequations}\label{eqn:nww}
      \begin{alignat}{2}
        \label{eqn:nww:r}
        \Psi_r              & \ge 0 && \ona \dell D, \\ 
        \label{eqn:nww:t}
        \Psi_\theta         & \ge 0 && \ona \dell D,   \\ 
        \label{eqn:nww:tt:r}
        \Psi_{\theta\theta} & \ge 0 && \ona R \cap \dell D,   \\ 
        \label{eqn:nww:tt:l}
        \Psi_{\theta\theta} & \le 0 && \ona L \cap \dell D.
      \end{alignat}
    \end{subequations}
    Now $(f,\Omega) \in U_2$ implies that $\Omega\bar{ \phi } +
    \tfrac 12 \Ca(\phi) \bar\phi \ne 0$, and so by \eqref{eqn:Psizw}
    \begin{align}
      \label{eqn:notboth}
      \tfrac 14 \big(\Psi_r^2 + r^{-2} \Psi_\theta^2 \big)
      = \abs{\dell_z\Psi}^2  
      = \abs{\Omega\bar{ \phi } +
      \tfrac 12 \Ca(\phi) \bar\phi} \ne 0.
    \end{align}
    Suppose for the sake of contradiction that $\Psi_r(z_0) = 0$ at some
    point $z_0 = \phi(e^{it_0}) \in \dell D$, in which case \eqref{eqn:notboth}
    forces $\Psi_\theta(z_0) \ne 0$. Writing $\phi =
    \rho e^{i\vartheta}$ as in Lemma~\ref{lem:polar},
    $\Psi|_{\dell D} \equiv 0$ gives
    \begin{align}
      \label{eqn:difft}
      0 = \frac d{dt} \Psi(\rho e^{i\vartheta})
      = \Psi_r \rho' + \Psi_{\theta} \vartheta'
      = \Psi_{\theta} \vartheta'
    \end{align}
    at $z_0$, and hence that $\vartheta' = 0$ there. But
    $(f,\Omega) \in U_3$ forces
    \begin{align*}
      \vartheta'(t) = \Re \left( \frac{e^{it}\phi'(e^{it})}{\phi(e^{it})} \right) > 0,
    \end{align*}
    so this is a contradiction.
    
    We now turn to \eqref{eqn:nww:t}--\eqref{eqn:nww:tt:l}. Since
    $\Psi_r > 0$ on $\dell D$, we know that $\dell D$ is a
    $C^{3+\alpha}$ curve which can be written in polar coordinates as
    $r=R(\theta)$. Thus the restrictions $\Psi^+ := \Psi|_D$ and
    $\Psi^- := \Psi|_{\C \without \overline D}$ are also $C^{3+\alpha}$ up
    to their respective boundaries.
    As in the proof of Lemma~\ref{lem:nodweak}, applying the strong
    maximum principle to $\Psi^\pm_\theta$ shows that $\Psi_\theta  >
    0$ on $S \without \dell D$. 
    
    We now study the values $\Psi_\theta$ on $S \cap \dell D$ using
    the Hopf lemma. Since $\Psi|_{\dell D} \equiv 0$ and $\Psi_r > 0$,
    an outward-pointing normal vector along $\dell D$ is
    \begin{align*}
      % \label{eqn:dn}
      \frac\dell{\dell n} 
      := \frac \dell{\dell r} + \frac 1{r^2} 
      \frac{\Psi_\theta}{\Psi_r} \frac \dell {\dell \theta}.
    \end{align*}
    Suppose for the sake of contradiction that $\Psi_\theta = 0$ at
    some point $z_0 \in S \cap \dell D$. Applying the Hopf lemma
    separately to $\Psi^\pm_\theta$, we find
    \begin{align}
      \label{eqn:contradictthis}
      \Psi^+_{\theta r} = 
      \frac \dell {\dell n} \Psi^+_\theta  < 
      0 < \frac \dell {\dell n} \Psi^-_\theta = \Psi^-_{\theta r}
    \end{align}
    at $z_0$. On the other hand, differentiating $\grad \Psi^+ = \grad
    \Psi^-$ along $\dell D$ yields
    \begin{align}
      \label{eqn:tcont}
      \Psi^+_{\theta\theta} - \frac{\Psi_\theta}{\Psi_r}  \Psi^+_{\theta r} 
      &=
      \Psi^-_{\theta\theta} - \frac{\Psi_\theta}{\Psi_r}  \Psi^-_{\theta r}, \\
      \label{eqn:rcont}
      \Psi^+_{r\theta} - \frac{\Psi_\theta}{\Psi_r}  \Psi^+_{r r} 
      &=
      \Psi^-_{r\theta} - \frac{\Psi_\theta}{\Psi_r}  \Psi^-_{r r}.
    \end{align}
    At $z_0$, $\Psi_\theta = 0$ and so \eqref{eqn:rcont} reduces to
    $\Psi^+_{r\theta}=\Psi^-_{r\theta}$, contradicting
    \eqref{eqn:contradictthis}. This completes the proof of
    \eqref{eqn:nw:t}.

    Finally, we treat \eqref{eqn:nw:tt:r}--\eqref{eqn:nw:tt:l} using
    the Serrin edge-point lemma~\cite[Appendix~E]{fraenkel:book}.
    Suppose for the sake of contradiction that $\Psi_{\theta\theta} =
    0$ at the singleton $z_0 \in \dell D \cap R$. Here we have dropped
    the superscript $\pm$ since $\Psi \in C^1(\C)$ and symmetry
    force $\Psi_{\theta\theta}^+ = \Psi^-_{\theta\theta}$ at $z_0$.
    Symmetry also forces
    $\Psi_{\theta r} = \Psi_{\theta rr} = \Psi_{\theta \theta\theta} =
    0$ along $R$. Considering $R \cap D$ as a lateral boundary of $S
    \cap D$, the Hopf lemma implies that $\Psi^+_{\theta\theta} >
    0$ there. Similarly we find that $\Psi^-_{\theta\theta} > 0$ on $R
    \without \overline D$. Thus we must have $\Psi^+_{\theta\theta r} \le
    0$ and $\Psi^-_{\theta\theta r} \ge 0$ at $z_0$. On the other hand,
    differentiating \eqref{eqn:rcont} along the boundary once more and
    plugging in $\Psi_{\theta\theta} = \Psi_{\theta r} = \Psi_{\theta
    rr} = \Psi_{\theta \theta\theta} = 0$ we eventually discover that
    $\Psi^+_{\theta\theta r} = \Psi^-_{\theta\theta r}$ at $z_0$, and so
    we must have $\Psi^+_{\theta\theta r} = \Psi^-_{\theta\theta r} =
    0$ there. But now we have shown that all of the first and second
    partials of $\Psi^+_\theta$ vanish at $z_0$, violating the Serrin
    edge-point lemma.
    The argument at $\dell D \cap L$ is similar.
  \end{proof}
\end{proposition}

We can now rule out the possibility of a loop by combining the
previous two lemmas.
\begin{proposition}[No loop]\label{prop:noloop}
  The nodal properties \eqref{eqn:n:glob} hold for all elements of
  $\cm$, i.e.~$\cm \sub \nodal$. Moreover,
  alternative~\ref{thm:global:loop} in Theorem~\ref{thm:global} does
  not occur and so \ref{thm:global:blowup} must occur.
  \begin{proof}
    With $\tilde f, \tilde \Omega$ as in Theorem~\ref{thm:global}, 
    set
    \begin{align*}
      s^* = \sup\{ \tilde s > 0 : 
      (\tilde f(s),\tilde \Omega(s)) \in \nodal 
      \text{ for $0 < s < \tilde s$} \}.
    \end{align*}
    By Lemma~\ref{lem:nodloc}, the set on the right hand side is
    nonempty. Assume for the sake of contradiction that $s^* <
    \infty$. By Lemma~\ref{prop:openclosed}, if $(\tilde f(s^*), \tilde
    \Omega(s^*)) \notin\triv$, then there is a neighborhood of
    $(\tilde f(s^*),\tilde \Omega(s^*))$ in $\soln \without \triv$
    which is contained in $\nodal$, a contradiction. So $(\tilde
    f(s^*),\tilde \Omega(s^*)) \in \triv$. Applying
    Lemma~\ref{lem:nodloc}, the only possibility is then $(\tilde
    f(s^*),\tilde \Omega(s^*)) = (0,\Omega_m)$. Appealing to
    Theorem~\ref{thm:locbif}\ref{thm:locbif:bif} and
    Lemma~\ref{lem:nodloc}\ref{lem:nodloc:bif}, this means that for $0
    < s < s^*$ the global curve $\cm$ has revisited portions of the
    local curve $\cm_\loc$ twice (once in either direction) without
    first revisiting the bifurcation point $(0,\Omega_m)$. This
    contradicts the analytic construction of $\cm$ in
    \cite[Theorem~9.1.1]{bt:analytic} or alternatively the 
    fact that $\cm$ has a local real-analytic reparametrization; see
    \cite[Proof of Theorem~5]{csv:critical}.
  \end{proof}
\end{proposition}

\subsection{Analyticity of the patch boundary}\label{sec:n:reg}

In Lemma~\ref{lem:extrareg} we showed that $\phi$ and hence also
$\dell D$ are smooth for every vortex patch in $\soln$. For patches
close to the unit disk, Castro, C\'ordoba, and
G\'omez-Serrano~\cite{ccg:reg} have shown that $\dell D$ and hence
also $\phi$ are in fact analytic. (Their argument also applies near
ellipses, and when the Euler equation is replaced by the generalized
Surface Quasi-Geostrophic equation.) In this section, we observe that
\emph{every} solution in $\soln$ is analytic. The proof relies on a
theorem of Kinderlehrer, Nirenberg, and Spruck~\cite{kns:freereg} for
elliptic free-boundary problems.

A consequence of Proposition~\ref{prop:noloop} is that every
solution along the global curve $\cm$ satisfies $\partial
\Psi/\partial n > 0$ on $\dell D$, where $n$ is a normal vector
pointing out of $D$. In the following lemma, we prove this more
generally.
\begin{lemma}\label{lem:simplemax}
  Let $\Psi,D,\Omega$ solve \eqref{eqn:Psi} with $D \in C^1$ and 
  $\Omega < 1/2$. Then
  \begin{align*}
    \frac{\partial \Psi}{\partial n} > 0 \ona \dell D.
  \end{align*}
  If $\Omega > 1/2$, then the reverse inequality holds.
  \begin{proof}
    First assume that $\Omega < 1/2$. From \eqref{eqn:Psi}, $\Psi$
    satisfies
    \begin{align*}
      \Delta \Psi = 1-2\Omega > 0 \ina D, \qquad \Psi = 0 \ona \dell
      D. 
    \end{align*}
    By the strong maximum principle, $\Psi$ therefore achieves its
    maximum over $\overline D$ on $\dell D$, where it is constant. Since
    $\dell D$ is $C^1$, by the Hopf lemma either $\partial
    \Psi/\partial n > 0$ at every point of $\dell D$ or $\Psi$ is
    constant in $D$. Since $\Delta \Psi > 0$ in $D$, $\Psi$ cannot be
    constant, and so the proof is complete. The argument for $\Omega >
    1/2$ is identical except that all of the inequalities are
    reversed.
  \end{proof}
\end{lemma}

% TODO : perhaps the notation here is too confusing.
\begin{theorem}[Analyticity of $\dell D$]\label{thm:freereg}
  Let $\Psi,D,\Omega$ solve \eqref{eqn:Psi} with $D \in C^1$. If $\Psi
  \in C^2(\C \without D) \cap C^2(\overline D)$, then $\dell D$ is
  analytic. 
  \begin{proof}
    If $\Omega = 1/2$, then $D$ is a disk by \cite{hmidi:trivial},
    and hence $\dell D$ is certainly analytic. So assume that
    $\Omega \ne 1/2$.
    Introducing the notation $\Omega^+ = D$, $\Omega^- = \C \without
    D$, and $\Gamma = \dell D$, we have that $\Psi \in
    C^1(\Omega^+ \cup \Omega^- \cup \Gamma) \cap C^2(\Omega^+ \cup
    \Gamma) \cap C^2(\Omega^- \cup \Gamma)$ satisfies the
    inhomogeneous linear elliptic equations
    \begin{alignat*}{2}
      F(z,\Psi,D\Psi,D^2\Psi) &:= \Delta \Psi + 2\Omega  = 0
      &\qquad& \ina \Omega^+,\\
      G(z,\Psi,D\Psi,D^2\Psi) &:= \Delta \Psi + 2\Omega-1 = 0
      && \ina \Omega^-.
    \end{alignat*}
    Moreover, Lemma~\ref{lem:simplemax} implies
    \begin{align*}
      \Psi = 0, \, \frac{\dell \Psi}{\dell n} \ne 0 \qquad \ona \Gamma.
    \end{align*}
    Thus, by \cite[Theorem~3.1]{kns:freereg}, $\Gamma = \dell D$ is
    analytic. 
  \end{proof}
\end{theorem}
We note that Theorem~\ref{thm:freereg} and Lemma~\ref{lem:simplemax}
have local versions where only part of $\dell D$ is assumed to be
$C^1$.
\begin{corollary}\label{cor:freereg}
  Let $(f,\Omega) \in \soln$ and let $D \sub \C$ be the associated
  vortex patch. Then $\dell D$ and $f$ are
  both real-analytic.
  \begin{proof}
    From the regularity of $f$ we know that $D \in C^{3+\alpha}$, and
    so standard elliptic theory gives $\Psi \in C^2(\overline D) \cap
    C^2(\C \without D)$. Thus Theorem~\ref{thm:freereg} applies and
    $\dell D$ is analytic. The conformal mapping $\Phi$ therefore
    extends to an holomorphic (and one-to-one) mapping on $\{\abs w >
    1 - \varepsilon\}$ for some $\varepsilon >
    0$~\cite[Proposition~3.1]{pommerenke:book}, and so in particular
    $t \mapsto f(e^{it}) = \Phi(e^{it}) - e^{it}$ is a real-analytic
    function of $t$ as desired.
  \end{proof}
\end{corollary}

\section{Uniform bounds}\label{sec:uniform}

We now turn our attention to the remaining alternative
\ref{thm:global:blowup} in Theorem~\ref{thm:global}. We will show that
the various quantities appearing in \eqref{eqn:global:blowup} can all
be controlled by the first and third terms, i.e.~by the relative fluid
speed and the tangent angle of the interface.

\subsection{Uniform regularity}
In this subsection we establish control over the H\"older norm
$\n\phi_{C^{1+\alpha}}$ appearing in \eqref{eqn:global:blowup}. The
first step is the following estimate from the theory of conformal
mappings.

\begin{lemma}[Koebe $1/4$ theorem]\label{lem:koebe}
  Any $(\phi-w,\Omega) \in U_1$ satisfies the bound
  $\n\phi_{L^\infty} \le 4$.
  \begin{proof}
    From Section~\ref{sec:space}, we know that $\phi$ extends to a
    conformal mapping $\Phi \maps \C \without \overline\D \to \C$.
    Consider the related conformal mapping $g \maps \overline\D
    \to \C$ defined by $g(\zeta) = 1/\Phi(1/\zeta)$. We easily check
    that $g(0) = 0$ and $g'(0) = 1$. Thus by the Koebe $1/4$ theorem
    (see for instance \cite[Theorem~1.3]{pommerenke:book}), we obtain
    $\abs{g(\zeta)} \ge 1/4$ for $\zeta \in \T$. Rewriting in terms of
    $\Phi$ we obtain $\abs{\Phi(w)} \le 4$ for $w \in \T$ as desired. 
  \end{proof}
\end{lemma}

We will also want to use the geometric information contained in the
condition $(\phi-w,\Omega) \in U_3$. For this we introduce the
notation
\begin{align}
  \label{eqn:beta}
  \gamma(w) = \arg \frac{w\phi'(w)}{\phi(w)}.
\end{align}
This is the quantity which appears in the definition \eqref{eqn:delta}
of $E_\delta^{1+\alpha}$, and represents (up to a sign) the tangent angle between
$\dell D$ and a circle. The next lemma states that a bound $\n
\gamma_{L^\infty} < \pi/2$ implies a bound on $\n{\phi'}_{L^p}$ for
some $p > 1$.
\begin{lemma}[\cite{gaier:nearly}] \label{lem:gaier}
  For $(\phi-w,\Omega) \in U_3$ and
  $\gamma$ defined in \eqref{eqn:beta} we have
  \begin{align}
    \label{eqn:gaier}
    \int_0^{2\pi} \abs{\phi'(e^{it})}^p  \, dt \le
    \frac{2\pi \cdot 4^p}{\cos(p\n\gamma_{L^\infty})}
    \qquad \fora
    0 \le p < \frac{\pi/2}{\n\gamma_{L^\infty}}.
  \end{align}
  \begin{proof}
    Recall from Section~\ref{sec:space} that $U_3 \sub U_1$, and that
    $(\phi-w,\Omega) \in U_1$ implies that $\phi$ extends to a conformal
    mapping $\Phi \maps \C \without \D \to \C$. Using $\Phi$, we define
    the holomorphic function
    \begin{align*}
      F(w) = \log \frac{w\Phi'(w)}{\Phi(w)} = u(w) + i\gamma(w)
    \end{align*}
    with real part $u$ and imaginary part $\gamma$. Since $F(\infty) =
    0$, the calculus of residues yields 
    \begin{align}
      \label{eqn:lookingood}
      1 = e^{pF(\infty)} = \frac 1{2\pi i} \int_\T
      e^{pF(w)} \frac{dw}w = \frac 1{2\pi} \int_0^{2\pi} e^{pu(e^{it})}
      \cos[p\gamma(e^{it})]\, dt
    \end{align}
    for any $p$. Assuming that $p \n\gamma_{L^\infty} < \pi/2$, we
    have $\cos(p\gamma(e^{it})) \ge \cos(p\n\gamma_{L^\infty})$ so
    that \eqref{eqn:lookingood} implies
    \begin{align}
      \label{eqn:combineme}
      \int_0^{2\pi} \frac{\abs{\Phi'(e^{it})}^p}{\abs{\Phi(e^{it})}^p}  \, dt =
      \int_0^{2\pi} e^{p u(e^{it})} \, dt \le
      \frac{2\pi}{\cos(p\n\gamma_{L^\infty})}.
    \end{align}
    Estimating $\abs{\Phi(e^{it})} \le 4$ using Lemma~\ref{lem:koebe} we
    are left with \eqref{eqn:gaier} as desired.
  \end{proof}
\end{lemma}

Next we need to obtain bounds for $\Omega\bar{ \phi } + \tfrac 12
\Ca(\phi) \bar\phi$ which, unlike those used in
Lemma~\ref{lem:extrareg}, do not depend on $C^{1+\alpha}$ bounds for
$\phi$. We will accomplish this by using Lemma~\ref{lem:stream} and
the following elliptic estimate for the relative stream function
$\Psi$.

\begin{lemma}[Basic elliptic estimate]\label{lem:ellipbasic}
  Let $(\phi-w,\Omega) \in \soln$ and fix $\beta \in (0,1)$. Then there
  exists a constant $C$ depending only $\beta$ so that the
  corresponding relative stream function $\Psi$ and vortex patch $D$
  satisfy
  \begin{align*}
    \n{\dell_z\Psi}_{C^\beta(\overline{D})} < C\abs{1-2\Omega}.
  \end{align*}
  \begin{proof}
    By Lemma~\ref{lem:koebe}, we know that
    $\n\phi_{L^\infty(\T)} \le 4$ and hence that $D \sub B_4$. By
    \eqref{eqn:Psi}, we know that $\Psi$ satisfies
    \begin{align*}
      \Delta \Psi = 1_D - 2\Omega \ina B_{10}
    \end{align*}
    in the sense of distributions, and that
    $\n{1_D-2\Omega}_{L^\infty(B_{10})} = \abs{1-2\Omega}$. Thus, for
    instance by \cite[Exercise~4.8]{gt}, we have
    \begin{align*}
      \n{\dell_z\Psi}_{C^{\beta}(B_5)} \le C(\beta) 
      \abs{1-2\Omega}.
    \end{align*}
    Since $\overline{D} \sub B_5$, this then implies the desired bound
    on $\n{\dell_z\Psi}_{C^\beta(\overline D)}$.
  \end{proof}
\end{lemma}

With the above lemmas in place, we can now establish the desired
bound for $\n\phi_{C^{1+\alpha}}$. Our hypothesis will be that the
first and third terms in \eqref{eqn:global:blowup} are controlled,
i.e.
\begin{align}
  \label{eqn:violate}
  \left\|\arg  \frac{w\phi'}{\phi}\right\|_{L^\infty} < \frac \pi 2 - \delta,
  \qquad 
  \inf_\T \abs{\Omega\bar{ \phi } + \tfrac 12 \Ca(\phi) \bar\phi} > \delta.
\end{align}

\begin{lemma}[Control of $\n\phi_{C^{1+\alpha}}$]\label{lem:uniform}
  Let $(\phi-w,\Omega) \in \soln$ with $\abs\Omega \le 10$ and
  suppose \eqref{eqn:violate} holds for some $\delta > 0$. Then there
  exists $C$ depending only on $\delta$ so that $\n
  \phi_{C^{1+\alpha}} < C$.
  \begin{proof}
    In what follows we use $C$ to denote any positive constant
    depending only on $\delta$. From Lemma~\ref{lem:gaier} we know
    that there exists $p > 1$ and depending only on $\delta$ such that
    $\n{\phi'}_{L^p} < C$. From Lemma~\ref{lem:koebe} we have
    $\n\phi_{L^\infty} < 4$, so Sobolev embedding gives
    $\n\phi_{C^\sigma} < C$ for some $\sigma$ depending on $p$. With
    $\beta \in (0,1)$ arbitrary but fixed, we also know by
    Lemma~\ref{lem:ellipbasic} that $\n{\dell_z\Psi}_{C^\beta(\overline D)} < C$.
   
    We now write $\F(\phi-w,\Omega) = 0$ as the Riemann--Hilbert
    problem $\Im(A\phi') = 0$, where by Lemma~\ref{lem:stream}
    \begin{align*}
      A :=  
      (\Omega\bar{ \phi } + \tfrac 12 \Ca(\phi) \bar\phi)w
      = 2\dell_z\Psi(\phi(w))w.
    \end{align*}
    Thanks to \eqref{eqn:violate}, we have 
    \begin{align*}
      \abs A =  \abs{\Omega\bar{ \phi } + \tfrac 12 \Ca(\phi) \bar\phi}
      = 2\abs{\dell_z\Psi(\phi(w))} > \delta.
    \end{align*}
    Lemmas~\ref{lem:rh} and \ref{lem:winding} yield the identity
    \begin{align}
      \label{eqn:fixedpointmaybe}
      \phi'(w) &=
      \exp \left\{
        \frac w{2\pi} \int_\T
        \frac 1{\tau-w}
        \left[
          \frac 1\tau \arg\bigg( \frac{\dell_z\Psi(\phi(\tau))\tau}{\bar{\dell_z\Psi(\phi(\tau)) \tau}}\bigg)
        -
          \frac 1w\arg\bigg( \frac{\dell_z\Psi(\phi(w))w}{\bar{\dell_z\Psi(\phi(w)) w}}\bigg)
        \right]
         d\tau \right\}.
    \end{align}
    From $\n \phi_{C^\sigma} < C$ and
    $\n{\dell_z\Psi}_{C^\beta(\overline D)} < C$, we get
    $\n{\dell_z\Psi \circ \phi}_{C^{\sigma\beta}(\T)} < C$. Since
    $\abs{\dell_z \Psi} > \delta/2$, it is then straightforward to
    show that 
    \begin{align*}
      \n[\bigg]{
        \frac 1\tau \arg\bigg( \frac{\dell_z\Psi(\phi(\tau))\tau}{\bar{\dell_z\Psi(\phi(\tau)) \tau}}\bigg)
      }_{C^{\sigma\beta}} 
      < C.
    \end{align*}
    The Cauchy integral is a bounded operator from $C^{\sigma\beta} \to
    C^{\sigma\beta}$, and so after composing with the exponential we
    obtain $\n{\phi'}_{C^{\sigma\beta}} < C$.

    In particular, we now know that $\n\phi_{C^{\sqrt\alpha}} < C$,
    and so we can repeat the above argument with $\sigma = \beta =
    \sqrt\alpha$ to obtain $\n{\phi'}_{C^\alpha} < C$ as desired.
  \end{proof}
\end{lemma}

\subsection{Other bounds}

We next turn to the other terms in \eqref{eqn:global:blowup}. First we
establish control over $\Omega$ by using the nonexistence results of
Hmidi~\cite{hmidi:trivial} and Fraenkel~\cite{fraenkel:book} together
with our result Proposition~\ref{prop:noloop} on nodal properties:
\begin{lemma}[Control of $\Omega$]\label{lem:omega}
  Along $\cm$, $0 < \Omega < 1/2$.
  \begin{proof}
    Suppose for the sake of contradiction that there exists a solution
    $(f,\Omega) \in \cm$ with $\Omega = 1/2$ or $\Omega = 0$. Then
    by \cite{hmidi:trivial} or Fraenkel~\cite{fraenkel:book} (as cited
    in \cite{hmidi:trivial}), we must have $f \equiv 0$,
    i.e.~$(f,\Omega) \in \triv$. But then $(f,\Omega) \notin
    \nodal$, contradicting Proposition~\ref{prop:noloop}.
  \end{proof}
\end{lemma}

Finally, we bound the two remaining quantities in
\eqref{eqn:global:blowup} in terms of the first and third.
\begin{lemma}[Remaining bounds on $\phi$]\label{lem:remaining}
  Let $(\phi-w,\Omega) \in \soln$, and suppose that
  \eqref{eqn:violate} holds for some $\delta > 0$. Then there exists a
  constant $C > 0$ depending only $\delta$ such that 
  \begin{align*}
    \abs{\phi'}
    ,\, 
    \abs{\phi} 
    \ge \frac 1C.
  \end{align*}
  \begin{proof}
    By Lemmas~\ref{lem:uniform} and \ref{lem:omega}, we have
    $\n\phi_{C^{1+\alpha}} < C$, where here and in what follows $C$ is
    a positive constant whose value may change from line to line but
    which depends only on $\delta$. 
    To get the lower
    bound on $\abs{\phi'}$, we take the multiplicative inverse of
    \eqref{eqn:fixedpointmaybe} and use our bounds on
    $\n\phi_{C^\alpha}$. Arguing as in the proof of
    Lemma~\ref{lem:uniform} we find that $\n{1/\phi'}_{C^\alpha} < C$
    and hence that $\min_\T \abs{\phi'} > 1/C$. 

    To get the lower bound on $\abs \phi$, we first get a lower bound
    on $\n\phi_{L^\infty}$ using the Schwarz lemma: Since the function
    $g(w) = \Phi(w)/w$ is holomorphic at infinity with $g(\infty) = 1$, we
    have by the maximum modulus principle that 
    \begin{align}
      \label{eqn:schwarz}
      \n \phi_{L^\infty(\T)} &= \n g_{L^\infty(\C\without D)} > 1.
    \end{align}
    Next we note that
    \begin{align*}
      \log \frac{\min_\T \abs\phi}{\max_\T \abs\phi}
      = \log \left| \frac{\phi(\pi/m)}{\phi(0)} \right|
      =-\Re \int_0^{\pi/m} \frac d{dt} \log \phi(e^{it})\, dt 
      = \int_0^{\pi/m} \Im \frac{e^{it} \phi'(e^{it})}{\phi(e^{it})}\, dt.
    \end{align*}
    Estimating this integral as in \eqref{eqn:combineme}, we obtain
    \begin{align*}
      \log \frac{\min_\T \abs\phi}{\max_\T \abs\phi}
      \le \int_0^{\pi/m} \frac{\abs{\phi'(e^{it})}}{\abs{\phi(e^{it})}}\, dt
      \le \frac{2\pi}{\cos(\pi/2-\delta)}.
    \end{align*}
    Taking exponentials gives
    \begin{align*}
      \min_\T \abs\phi \ge \max_\T \abs\phi 
    \exp\left( -
      \frac{2\pi}{\cos(\pi/2-\delta)}\right) 
      > 
    \exp\left( -
      \frac{2\pi}{\cos(\pi/2-\delta)}\right) > 1/C,
    \end{align*}
    where in the second-to-last inequality we have used
    \eqref{eqn:schwarz}.
  \end{proof}
\end{lemma}

\subsection{Proof of Theorem~\ref{thm:informal}}%\label{sec:proof}

We are now ready to prove our main result, Theorem~\ref{thm:informal}.
\begin{proof}[Proof of Theorem~\ref{thm:informal}]
  Conclusion \ref{thm:informal:bif} of Theorem~\ref{thm:informal} is
  immediate from the construction thus far: Theorem~\ref{thm:global}
  constructed $\cm$ as an extension of $\cm_\loc$, which indeed starts
  at the circular patch $(0,\Omega_m)$; see
  Theorem~\ref{thm:locbif}\ref{thm:locbif:bif} and
  Definition~\ref{def:loc}. As mentioned at the start of
  Section~\ref{sec:nodal}, the conclusion \ref{thm:informal:nodal} of
  the theorem is implied by \eqref{eqn:n:glob}, which holds by
  Proposition~\ref{prop:noloop}. Since \ref{thm:informal:reg}
  follows from Corollary~\ref{cor:freereg}, it therefore remains to
  show \ref{thm:informal:vanish}.

  By Proposition~\ref{prop:noloop},
  alternative~\ref{thm:global:loop} in Theorem~\ref{thm:global} does
  not occur. Therefore alternative~\ref{thm:global:blowup} occurs,
  that is 
  \begin{align*}
    \min\left\{
      \min_\T \abs[\big]{ \tilde\Omega\bar{\tilde \phi } + \tfrac 12 \Ca(\tilde\phi) \bar{\tilde\phi}}
      ,\,
    \frac 1{1+\abs{\tilde\Omega}+\n{\tilde \phi}_{C^{1+\alpha}}},\,
    \frac \pi 2 - \max_\T \bigg| \arg \frac{w\tilde \phi'}{\tilde\phi} \bigg|,\,
    \min_\T \abs{\tilde\phi'},\,
    \min_\T \abs{\tilde\phi} \right\} \longrightarrow 0
  \end{align*}
  as $s \to \infty$. Applying Lemmas~\ref{lem:uniform},
  \ref{lem:omega}, and \ref{lem:remaining}, we see that this implies
  the simpler condition
  \begin{align}
    \label{eqn:nearlythere}
    \min\left\{
      \min_\T \abs[\big]{ \tilde\Omega\bar{\tilde \phi } + \tfrac 12 \Ca(\tilde\phi) \bar{\tilde\phi}}
    ,\,
    \frac \pi 2 - \max_\T \bigg| \arg \frac{w\tilde \phi'}{\tilde\phi} \bigg|
    \right\} \longrightarrow 0
  \end{align}
  as $s \to \infty$.
  Letting $\tilde\Psi(s)$ be the relative stream function and $\tilde
  D(s)$ the vortex patch associated to $(\tilde f(s),\tilde
  \Omega(s))$, we claim that 
  \begin{align*}
    0 < \min_{\dell\tilde D} \tilde \Psi_r \to 0
  \end{align*}
  as $s \to \infty$. The left inequality is just a restatement of the
  nodal property \eqref{eqn:n:r}. Suppose for the sake of
  contradiction that $\tilde \Psi_r > \varepsilon$ on $\dell \tilde D$
  along some subsequence $s_n \to \infty$ for some $\varepsilon > 0$.
  Then the first term in \eqref{eqn:nearlythere} is
  $\abs{\grad\tilde\Psi}^2/4 \ge \varepsilon^2/4$ (see
  \eqref{eqn:Psizw} or \eqref{eqn:notboth}), and so
  \begin{align}
    \label{eqn:cannothappen}
    \max_\T\bigg|\arg \frac {w\tilde\phi'}{\tilde\phi}\bigg| \to \frac \pi 2
    \asa n \to \infty.
  \end{align}
  Differentiating $\tilde \Psi \circ \tilde \phi \equiv 0$ as in
  \eqref{eqn:difft} and using Lemma~\ref{lem:polar}, we see that 
  \begin{align*}
    \abs{\tilde\phi} \tan\bigg(\arg \frac
    {w\tilde\phi'}{\tilde\phi}\bigg)
    = \frac
    {\abs\phi\Im(w\tilde\phi'/\tilde\phi)}
    {\Re(w\tilde\phi'/\tilde\phi)}
    = -\frac{\tilde\rho'}{\tilde \vartheta'}
    = \frac{\tilde\Psi_\theta}{\tilde\Psi_r},
  \end{align*}
  and hence that
  \begin{align*}
    \bigg\|\arg \frac
    {w\tilde\phi'}{\tilde\phi}\bigg\|_{L^\infty(\T)}
    = 
    \tan^{-1}\bigg\|\frac{\tilde\Psi_\theta/r}{\tilde\Psi_r}
    \bigg\|_{L^\infty(\dell\tilde D)}
    \le 
    \tan^{-1}\Big(\varepsilon^{-1}\n{\tilde\Psi_\theta/r}_{L^\infty(\dell\tilde D)}\Big)
  \end{align*}
  for all $n$. Therefore the only way for \eqref{eqn:cannothappen} to
  occur is if $\n{\tilde\Psi_\theta/r}_{L^\infty(\dell\tilde D)} \to
  \infty$. But by Lemma~\ref{lem:ellipbasic}
  $\n{\dell_z\tilde\Psi}_{C^{1/2}(\tilde D)}$, say, is bounded along
  $\cm$, and hence this is impossible.
\end{proof}

\section{Numerical streamline patterns}\label{sec:numerical}

In this section we numerically calculate global branches of rotating
vortex patches with $m=3,4,5,6$. As mentioned in the introduction,
similar branches have previously been calculated
in~\cite{dz:vstates,woz:numerical,overman:limiting}, with the striking
conclusion that there are limiting solutions with sharp \ang{90}
corners. Our contribution is that we additionally calculate the full
stream function $\Psi$; the results suggest that the qualitative
features in Theorem~\ref{thm:phase} persist along the whole branch
(Conjecture~\ref{conj:nodal}) and indeed are related to the formation
of sharp corners.

\subsection{Numerical method}

We approximate the trace $\phi$ of the conformal mapping $\Phi$ by a
Fourier series with $M$ modes:
\begin{align*}
  \phi(e^{it}) \approx e^{it} + \sum_{n=1}^{M} a_n e^{-in(m-1)t},
\end{align*}
where $a_1,\ldots,a_M$ are real. This is in line with the
normalization $\Phi'(\infty) = 0$ in Section~\ref{sec:space}. With an
(inverse) fast Fourier transform, the values of $\phi(e^{it})$ and
$\phi'(e^{it})$ are then approximated at $N > mM$ evenly spaced values
$t_1,\ldots,t_N$. These physical grid points $z_n = \phi(e^{it_n})$
become denser in the regions where $\dell D$ has the high curvature.
The integral appearing in \eqref{eqn:phi} is then approximating with
the trapezoid rule,
\begin{align*}
  \int_\T 
  \frac{\bar{\phi(\tau)}-\bar{\phi(e^{it_n})}} { \phi(\tau) -
  \phi(e^{it_n})} \phi'( \tau)\, d\tau 
  \approx
  \bar{ie^{it}\phi'(e^{it})} \frac{2\pi}N + 
  \sum_{k\ne n}
  \frac{\bar{\phi(e^{it_k})}-\bar{\phi(e^{it_n})}} { \phi(e^{it_k}) -
  \phi(e^{it_n})} \phi'(e^{it_k})  \frac{2\pi ie^{it_k}}N,
\end{align*}
where we have evaluated the integrand at $\tau = e^{it_n}$ by
calculating the limit
\begin{align*}
  \lim_{t \to s} 
  \frac{\bar{\phi(e^{is})}-\bar{\phi(e^{it})}} { \phi(e^{is}) -
  \phi(e^{it_n})} \phi'( e^{is})
  = -\bar{\phi'(e^{is})}{e^{2is}}.
\end{align*}
Substituting these approximations into \eqref{eqn:F} yields an
approximation of $\F(\phi-w,\Omega)(e^{it})$ for $t=t_1,\ldots,t_N$.
Taking a fast Fourier transform, we obtain
\begin{align*}
  \F(\phi-w,\Omega)(e^{it}) \approx \Im \sum_{n=1}^M b_n e^{inmt}
\end{align*}
for real coefficients $b_1,\ldots,b_M$. This process defines a
finite-dimensional mapping
\begin{align*}
  \F^M \maps \R^M \by \R \to \R^M,
  \qquad 
  \F^M(a_1,\ldots,a_M;\Omega) = (b_1,\ldots,b_M),
\end{align*}
whose roots correspond to rotating vortex patches.

We find roots of $\F^M$ by using a standard Newton--Krylov scheme.
This is a Newton-type method in which the action of the Jacobian
matrix $\F^M_a(a;\Omega)$ of $\F^M$ is approximated by
\begin{align*}
  \F^M_a(a;\Omega)\alpha \approx \frac{\F^M(a+\varepsilon 
  \alpha;\Omega)-\F^M(a;\Omega)}\varepsilon
\end{align*}
so that the action of the inverse matrix can in turn be
approximated using an iterative LGMRES method. Fixing $\delta > 0$, we
seek solutions at the discrete frequencies $\Omega^k = \Omega_m -
k\delta$. We begin tracing out the branch by computing a solution
$a^1$ of $\F^M(a^1;\Omega^1) = 0$ using $a^1=a^0:=(\sqrt\delta,0,\ldots,0)$ as
our initial guess. We then solve $\F^M(a^2;\Omega^2) = 0$ for $a^2$
using $a^1$ as an initial guess and so on. The process terminates when
too many Newton iterations are needed or else when the solutions are no
longer well-resolved by $M$ Fourier modes.

\subsection{Results}

\begin{table}
  \centering
  \begin{tabular}{r|cccc}
    Symmetry class $m$ & 3 & 4 & 5 & 6 \\ \hline
    Fourier modes $M$ & 1023 & 1023 & 511 & 255 \\
    Physical gridpoints $N$ & 6144 & 8192 & 5120 & 3072
  \end{tabular}
  \caption{Number of Fourier modes and gridpoints used for different
  values of $m$.}
  \label{fig:modes}
\end{table}
\begin{figure}
  \centering
  \includegraphics[scale=.635]{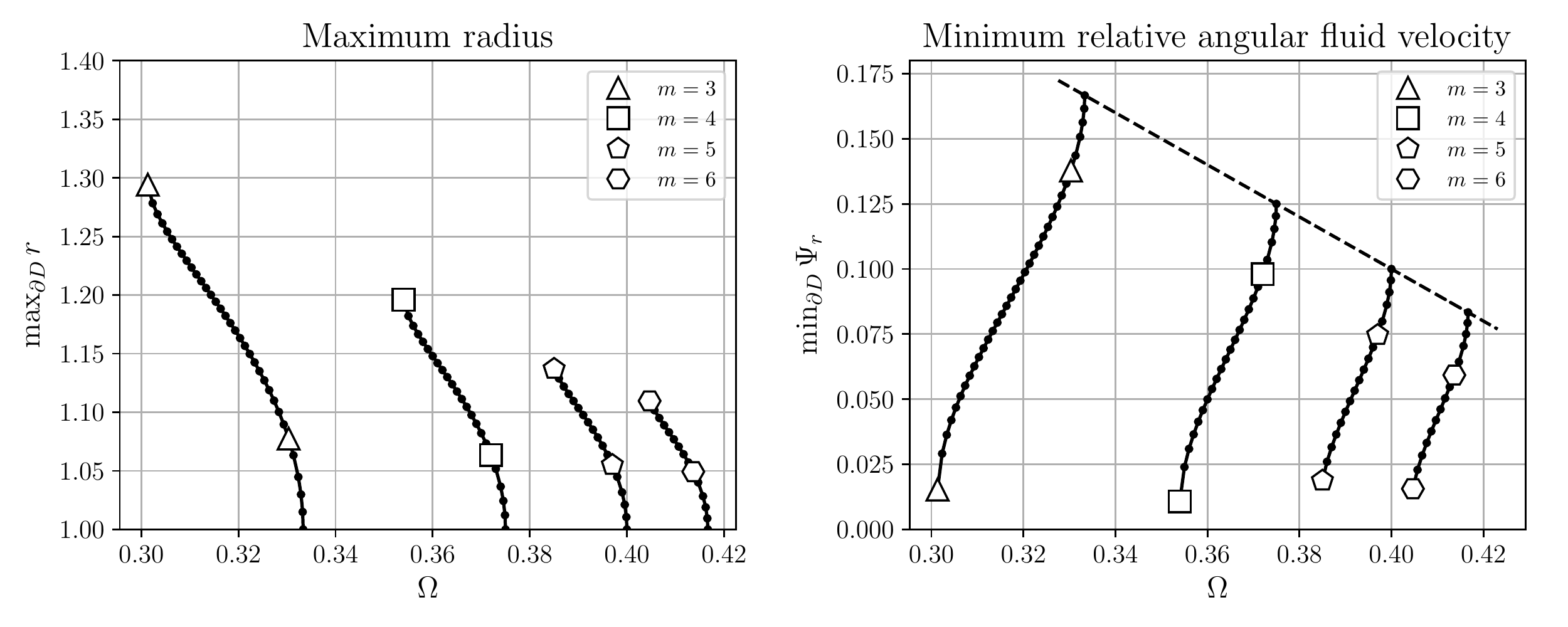}
  \caption{Maximum radius and minimum (relative) angular velocity
  along $\dell D$ for the numerically computed branches. The solutions
  with markers appear in Figures~\ref{fig:smallish} and
  \ref{fig:large}, and the dashed line is the analytical formula for
  trivial solutions.}
  \label{fig:branches}
\end{figure}
\begin{figure}
  \centering
  \includegraphics[width=.7\textwidth]{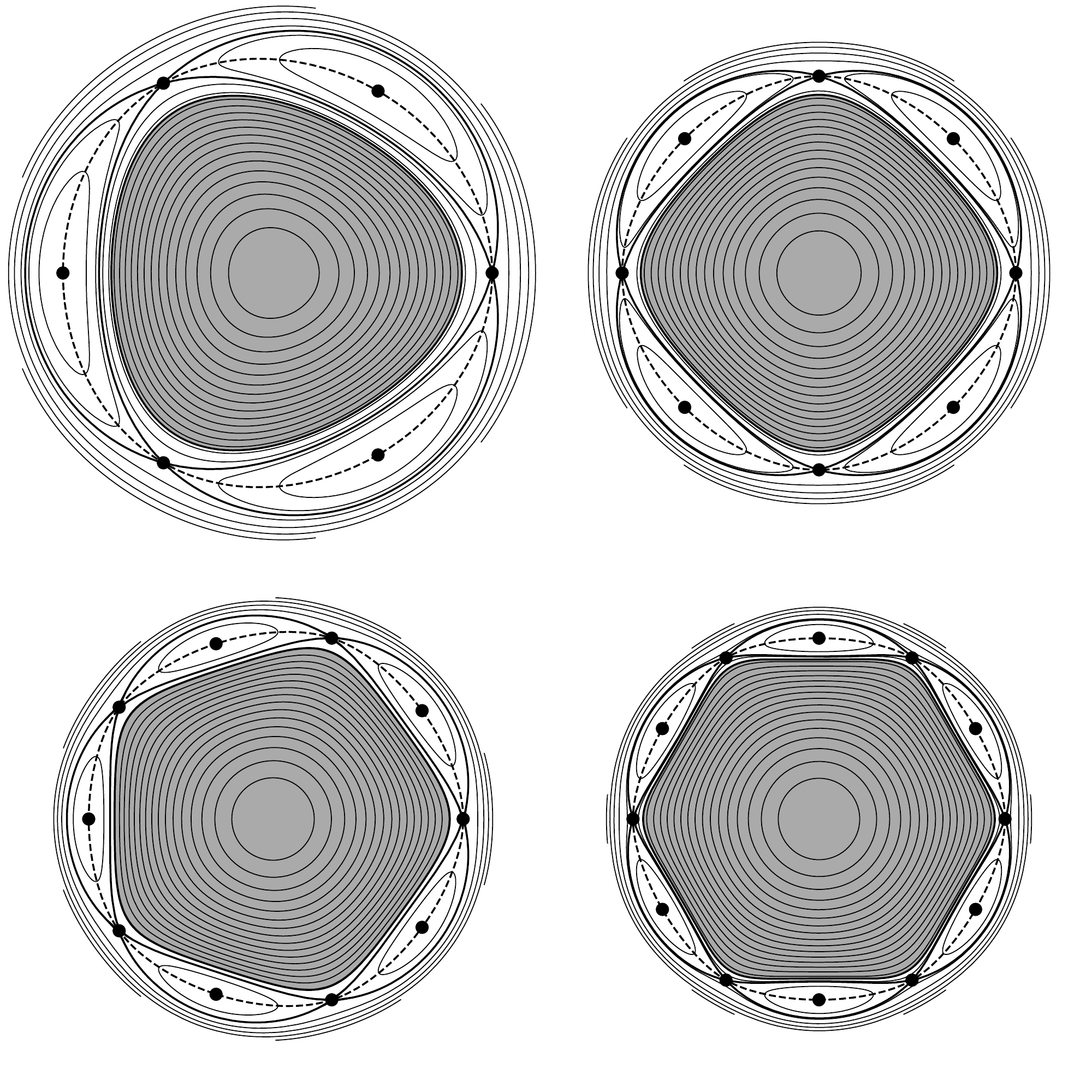}
  \caption{For each numerical branch in Figure~\ref{fig:branches},
  the level curves of the stream function $\Psi$ for the solution
  with the fifth-largest value of $\Omega$. The shaded region is the
  patch $D$, the dashed lines are the curves $\Psi_r = 0$ where there
  is no angular fluid velocity, and the markers are the critical
  points of $\Psi$ (besides the origin).
  }
  \label{fig:smallish}
\end{figure}
\begin{figure} 
  \centering
  \includegraphics[width=.7\textwidth]{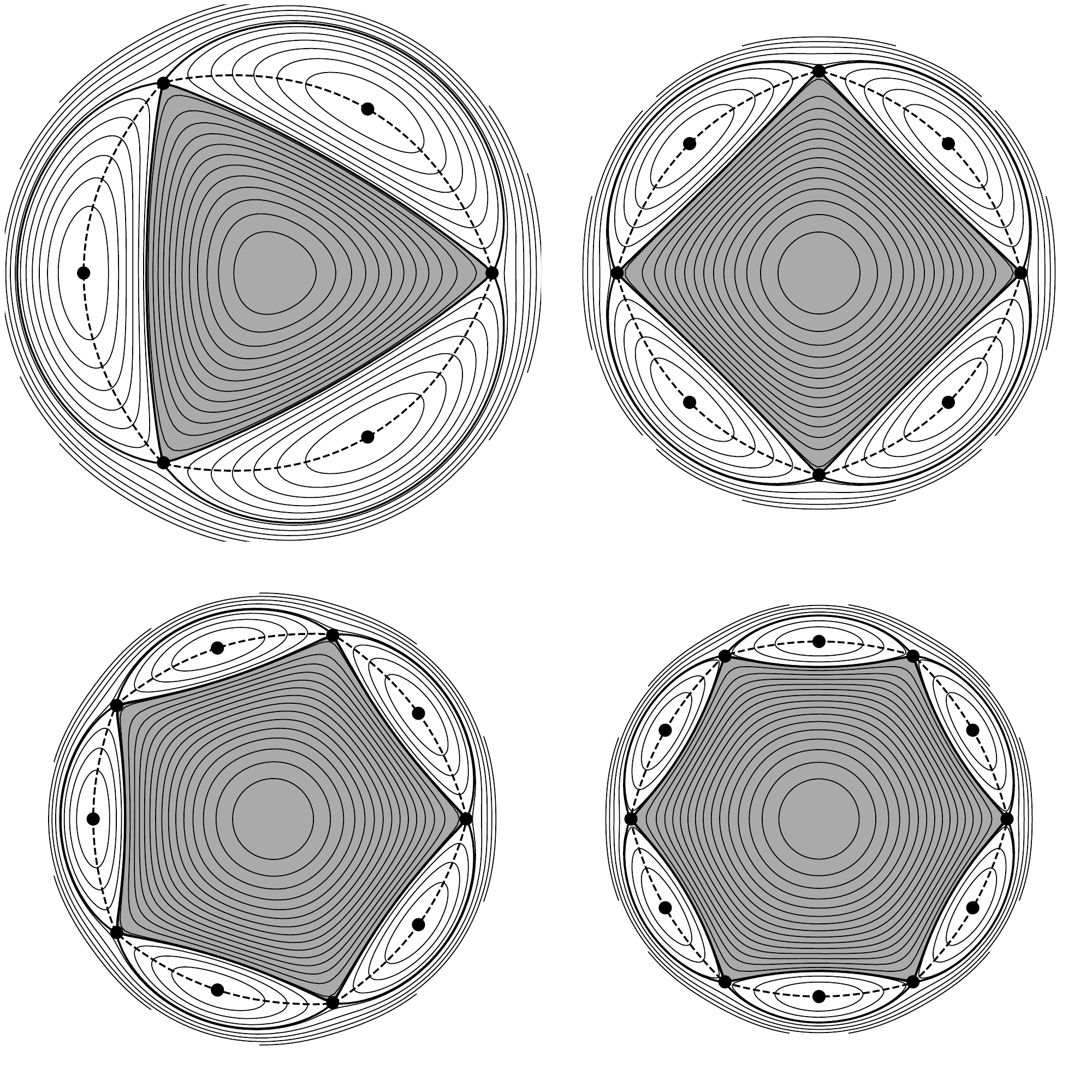}
  \caption{For each numerical branch in Figure~\ref{fig:branches}, the
  level curves of the stream function $\Psi$ for the solution with the
  smallest value of $\Omega$. The shaded region is the patch $D$, the
  dashed lines are the curves $\Psi_r = 0$ where there is no angular
  fluid velocity, and the markers are the critical points of $\Psi$
  (besides the origin). }
  \label{fig:large}
\end{figure}
We applied the above method with $m=3,4,5,6$, a grid spacing of
$\delta = 0.001$ in $\Omega$, and the numbers $M$ of Fourier
modes and $N$ of gridpoints given in Table~\ref{fig:modes}. To better
resolve the local bifurcation, we computed two additional solutions
near the start of each branch.

The maximum radii $\max_{\dell D} r = \n\phi_{L^\infty}$ are shown in
Figure~\ref{fig:branches}. As seen in previous work, the angular
frequency $\Omega$ appears to be decreasing along each branch while
the radius increases. Theorem~\ref{thm:informal} predicts that
$\min_{\dell D}\Psi_r$ should also limit to zero, and evidence of
this can indeed be seen in Figure~\ref{fig:branches}.

Now we turn to Conjecture~\ref{conj:nodal} on the level curves of the
relative stream function $\Psi$. Looking at the proof of
Theorem~\ref{thm:phase}, the conjecture is true provided the
inequalities
\begin{align*}
   \Psi_{r\theta}^- < 0 \ona S \cap \dell D,
   \quad
   \Psi^-_{r\theta\theta} < 0 \ona R \cap \dell D,
   \quad
    \Psi^-_{r\theta\theta} > 0 \ona L \cap \dell D
\end{align*}
and also
\begin{align*}
  \max_{\dell D} \left( (r\dell_r)^2\Psi^- + 2\Omega r^2 \right) &<
  \min_{\dell D} 2\Omega r^2
\end{align*}
hold for all solutions. Thanks to Lemma~\ref{lem:stream}, the above
quantities are readily calculated in terms of $\phi$, and the
inequalities do seem to hold for all of the solutions we computed.

\begin{figure}
  \centering
  \includegraphics[width=\textwidth]{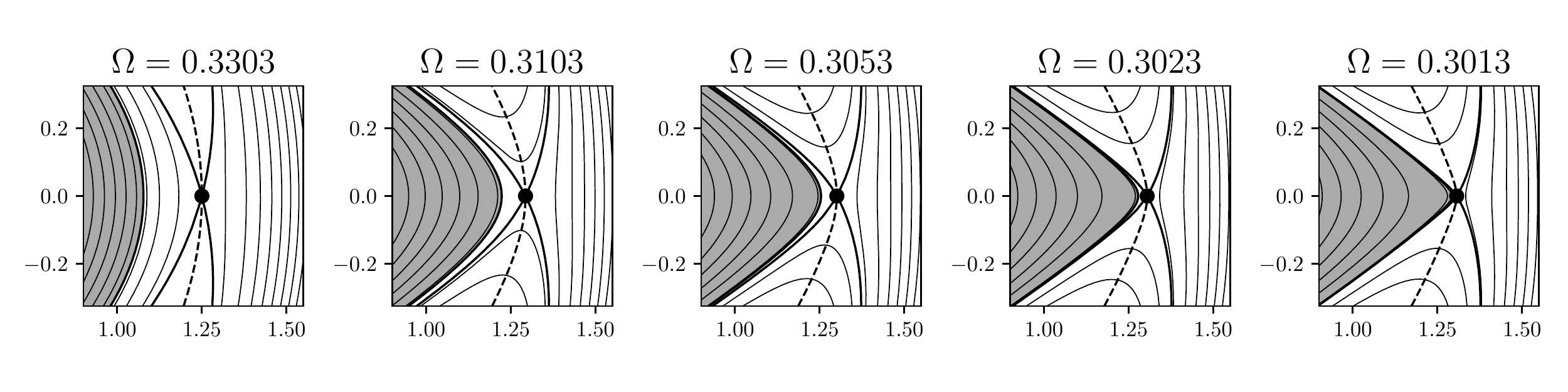}
  \caption{Streamline plots of several patches with $m=3$, magnified
  to show the approach of the saddle point to $\dell D$. As in
  Figures~\ref{fig:smallish} and \ref{fig:large}, the shaded region is
  $D$, the dashed line is the curve $\Psi_r = 0$, and the marked point
  is the saddle.}
  \label{fig:stagmove}
\end{figure}
\begin{figure} 
  \centering
  \includegraphics[scale=.635]{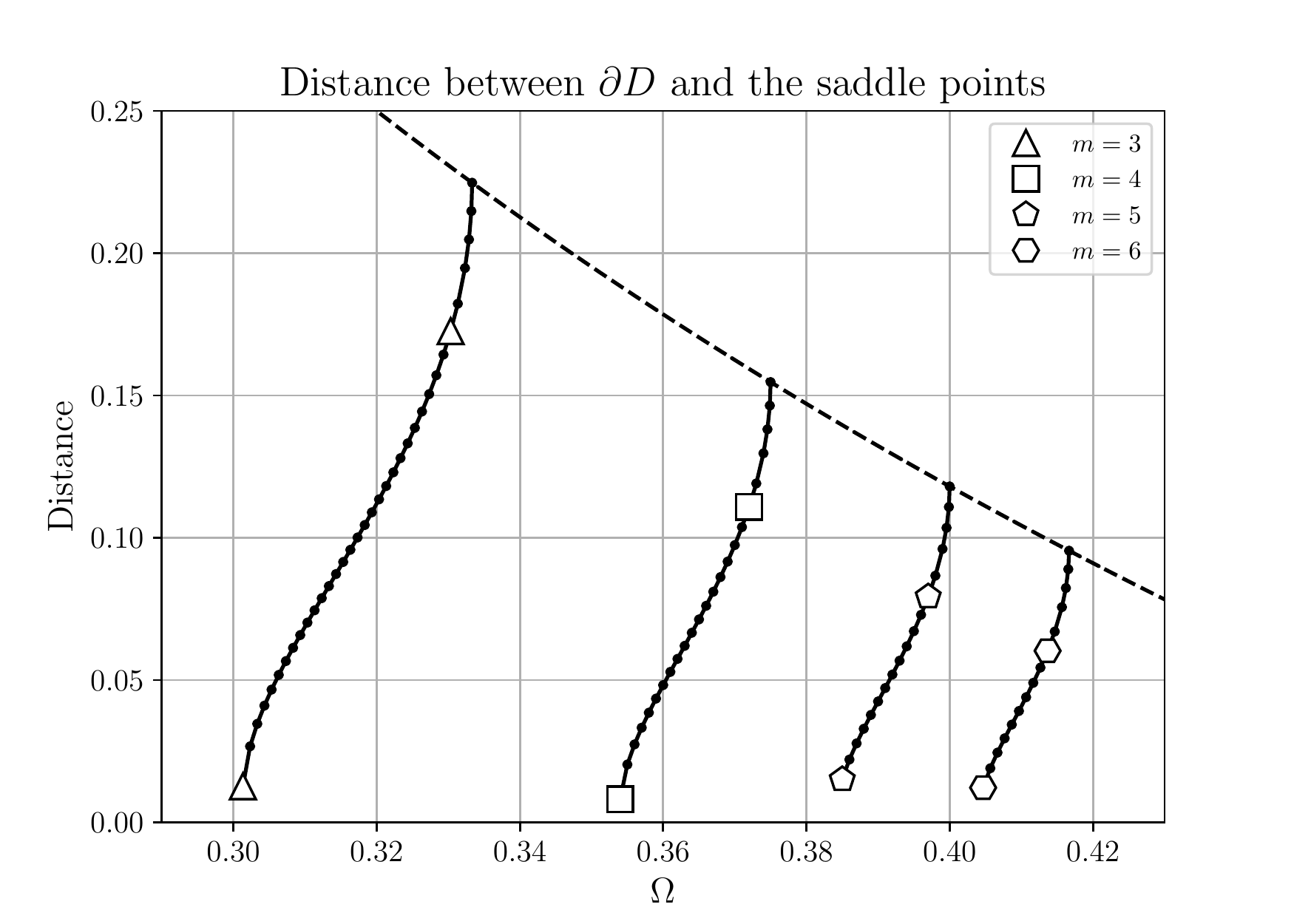}
  \caption{The distance between the closest saddle point and $\dell D$ for the
  numerical branches from Figure~\ref{fig:branches}. The dashed curve
  is the analytical formula for trivial solutions.}
  \label{fig:saddles}
\end{figure}
Using the formulas from Appendix~\ref{sec:stream}, we can also compute
and plot the full stream function $\Psi$ for any of our numerical
solutions. Level curves of $\Psi$ for several solutions are shown in
Figures~\ref{fig:smallish} and \ref{fig:large}. The solutions in
Figure~\ref{fig:smallish} are part of the way up our numerical
branches, while the solutions in \ref{fig:large} are at the very end.
As expected, all of these phase portraits have the qualitative
features from Theorem~\ref{thm:phase}. At the end of each branch,
the saddle points approach the boundary $\dell D$; see
Figures~\ref{fig:stagmove} and \ref{fig:saddles}. It would seem that,
in the limit, the saddle point coincides with the corner point on
$\dell D$, which in turn is made up of heteroclinics from
Theorem~\ref{thm:phase}.

\FloatBarrier
\appendix
\section{Appendix}
\subsection{Linear Riemann--Hilbert problems}\label{sec:rh}

In this section we prove Lemma~\ref{lem:rh} by recalling some basic
facts from the classical theory of Riemann--Hilbert problems. As
mentioned in Section~\ref{sec:space}, the space $X^{k+\beta}$
in Lemma~\ref{lem:rh} consists of the traces $f=F|_\T$ of
certain holomorphic functions $F$ on the exterior of the unit disk
$\D$. In this section it is more convenient to work with the full
mappings $F$. 

To cite the classical theory, we first consider holomorphic functions
on the \emph{interior} of the unit disk, introducing the (real) Banach spaces
\begin{align*}
  \mathcal Z_1 = \big\{ F_1 \in C^{k-1+\beta}(\overline\D,\C) : \text{$F_1$ is
  holomorphic on $\D$}  
  \big\},
  \qquad 
  \mathcal Y = C^{k+\beta}(\T,\R),
\end{align*}
and fixing a $C^{k-1+\beta}(\T,\C)$ coefficient function $a_1$ with winding
number $0$. From \cite[\S40]{muskhelishvili}, we have the following
results for the linear operator $\mathcal L_1 \maps \mathcal Z_1 \to
\mathcal Y$ defined by $\mathcal L_1F_1 = \Re\{a_1 F_1|_\T\}$. 
\newpage % XXX
\begin{theorem}[Properties of $\mathcal L_1$]\hfill
  \label{thm:musk}
  \begin{enumerate}[label=\rm(\alph*)]
  \item $\ker \mathcal L_1 = \Span\{F_1^0\}$ is one-dimensional, where
    for $w \in \D$
    \begin{align*}
      % \label{eqn:F0}
      F_1^0(w) = 
      \exp \left\{ 
      - \frac 1{4\pi} \int_\T \frac{\theta_1(\tau)}\tau \, d\tau 
      + \frac 1{2\pi} \int_\T
      \frac {\theta_1(\tau)}{\tau-w}
      \, d\tau 
      \right\},
    \qquad 
      \theta_1(w) :=  \arg\bigg( -\frac{\bar {a_1(w)}}{a_1(w)}\bigg).
    \end{align*}
  \item For $h \in \mathcal Y$, the general solution of $\mathcal
    L_1 F_1 = h$ is given by
  \begin{align*}
    % \label{eqn:Fgen}
    \begin{aligned}
      F_1(w) &= \frac{F_1^0(w)}{\pi i} 
      \int_\T \frac 1{\tau - w}\frac {h(\tau)}{a_1(\tau) F_1^0(\tau)} \, d\tau\\
      &\qquad -\frac{F_1^0(w)}{2\pi i} 
      \int_\T \frac {h(\tau)}{a_1(\tau) F_1^0(\tau) \tau} \, d\tau
      + CF_1^0(w)
    \end{aligned}
    \qquad \fora w \in \D,
  \end{align*}
  where $C$ is an arbitrary real constant. In particular, $\mathcal
  L_1$ is Fredholm with index 1. 
  \end{enumerate}
\end{theorem}

Using a simple change of variables, we now obtain an analogue of
Theorem~\ref{thm:musk} with the interior of the unit disk replaced by
the exterior and where imaginary instead of real parts are taken. The
relevant analogue of $\mathcal Z_1$ is 
\begin{align*}
  \mathcal Z_2 = \left\{ f \in C^{k-1+\beta}(\C \without \D,\C) : \text{$f$ is
  holomorphic on $\C \without \overline\D$, bounded at $\infty$}  
  \right\},
\end{align*}
and we define $\mathcal L_2 \maps \mathcal Z_2 \to
\mathcal Y$ by $\mathcal L_2 f = \Im\{a_2 F_2|_\T\}$. As before we
assume that $a_2 \in C^{k-1+\beta}(\T,\C)$ and has zero winding number.
With the identifications
\begin{align}
  \label{eqn:rhvars}
  a_1(w) = i\bar{a_2(w)},
  \qquad 
  F_1(w) = \bar{F_2(1/\bar w)},
\end{align}
one can easily check that $\mathcal L_1 F_1 = h$ if and only if $\mathcal L_2 F_2 = h$.
Moreover the mapping $F_2 \mapsto F_1$ is linear and invertible $\mathcal
Z_2 \to \mathcal Z_1$. This implies the following corollary of
Theorem~\ref{thm:musk}:
\begin{corollary}[Properties of $\mathcal L_2$]\hfill
  % \label{cor:musk}
  \begin{enumerate}[label=\rm(\alph*)]
  \item $\ker \mathcal L_2 = \Span\{F_2^0\}$ is one-dimensional, where for 
    $\abs w > 1$
    \begin{align}
      \label{eqn:F02}
      F_2^0(w) = 
      \exp \left\{ 
      \frac 1{4\pi} \int_\T \frac{\theta_2(\tau)}\tau
      \,d\tau
      + \frac w{2\pi} \int_\T
      \frac {\theta_2(\tau)}{\tau-w}
      \, \frac {d\tau}\tau
      \right\},
      \qquad 
      \theta_2(w) :=  \arg\bigg( \frac{a_2(w)}{\bar{a_2(w)}}\bigg).
    \end{align}
  \item For $h \in \mathcal Y$, the general solution of $\mathcal
    L_2 F_2 = h$ is given by, for $\abs w > 1$,
  \begin{align}
    \label{eqn:Fgen2}
    F_2(w) &= 
    -\frac{wF_2^0(w)}\pi 
    \int_\T \frac 1{\tau-w}
    \frac {h(\tau)}{a_2(\tau) F_2^0(\tau)} 
    \, \frac{d\tau}\tau
    -\frac{F_2^0(w)}{2\pi} 
    \int_\T \frac {h(\tau)}{a_2(\tau) F_2^0(\tau) } \,
    \frac{d\tau}\tau
    + CF_2^0(w),
  \end{align}
  where $C$ is an arbitrary real constant. 
  In particular, $\mathcal
  L_2$ is Fredholm with index 1. 
  \end{enumerate}
  \begin{proof}
    Make the change of variables \eqref{eqn:rhvars}, and use the
    substitution $\tau \mapsto 1/\tau$ to rewrite the integrals in
    \eqref{eqn:F02} and \eqref{eqn:Fgen2}.
  \end{proof}
\end{corollary}

Assuming that our coefficient $a_2$ has the symmetries $a_2(\bar w) =
\bar{a_2(w)}$ and $a_2(e^{2\pi i/m} w) = a_2(w)$, 
it is easy to show that $\mathcal L_2$ restricts to a map $\mathcal L_3
\maps \mathcal Z_3 \to Y$ with the same kernel and Fredholm
index, where 
\begin{align*}
  \mathcal Z_3 &= \left\{ F \in \mathcal Z_2 
  : F(\bar w) = \bar{F(w)},\, F(e^{2\pi i/m} w) = F(w) \right\}.
\end{align*}
These symmetries of $a_2$ also imply
\begin{align}
  \label{eqn:simplifications}
  \int_\T \frac{\theta_2}\tau d\tau = 0,
  \qquad 
  \int_\T \frac h{a_2F_2^0\tau}\, d\tau =0
\end{align}
for all $h \in Y$, allowing for the formulas
\eqref{eqn:F02} and \eqref{eqn:Fgen2} to be simplified.

Finally we let
\begin{align*}
  \mathcal X &= \left\{ G \in C^{k+\beta}(\C \without \D,\C) : 
  \text{$G$ is holomorphic on $\C \without \overline\D$, 
  bounded at $\infty$}  \right\}
\end{align*}
be a subspace of $\mathcal Z_2$ with additional regularity, and 
\begin{align*}
  \mathcal X_3 &= \left\{ G \in \mathcal X : 
  G(e^{2\pi i/m}w) = e^{2\pi i/m} G(w)
  ,\, 
  G(\bar w) = \bar{G(w)}
  ,\,
  G'(\infty) = 0 \right\}
\end{align*}
be a further subspace with additional symmetries, to be compared with
those inherent in the definition of $X^{k+\beta}$. The derivative operator
$\tfrac d{dw}$ is invertible $\mathcal X_3 \to \mathcal Z_3$. Indeed,
the only potential complication is the uniqueness of inverses, and
this holds thanks to the constraint $G'(\infty) = 0$ in the definition of
$\mathcal X_3$ (equivalently, the Laurent series of a function in
$\mathcal X_3$ has no constant term). 
\begin{proof}[Proof of Lemma~\ref{lem:rh}]
  Let $\tilde X^{k+\beta} = \Span\{w\} + X^{k+\beta}$, and let
  $\mathcal T \maps \mathcal X_3 \to \tilde X^{k+\beta}$ be the trace
  operator, $\mathcal Tg = g|_\T$. We easily check that $\mathcal T$ is invertible, and
  consider the composite mapping 
  \begin{align*}
    \mathcal S = \mathcal L_3 \frac d{dw} \mathcal T^{-1} \maps \tilde X^{k+\beta} \longrightarrow Y^{k-1+\beta}.
  \end{align*}
  Our above arguments show that $\mathcal S$ is Fredholm with index
  $1$ and that its kernel is spanned by the function $g_0 \in \tilde X^{k+\beta}$
  characterized by $g_0' = F_2^0|_\T$. Using the usual formulas for
  the limiting values of Cauchy integrals (e.g. \cite[Equation
  16.4]{muskhelishvili}), and remembering the cancellation
  \eqref{eqn:simplifications}, we obtain
  \begin{align*}
    % \label{eqn:g0}
    g_0'(w)
    = \exp \left\{ 
      \frac w{2\pi} \int_\T
      \frac {\tau^{-1}\theta_2(\tau)-w^{-1}\theta_2(w)}{\tau-w}
      \, d\tau
      \right\}.
  \end{align*}
  Furthermore, $F_2^0(\infty) = 1$ implies that the coefficient of $w$
  in the Fourier series of $g_0(w)$ is also $1$, and hence that $g_0 -
  w \in X^{k+\beta}$. Thus $g_0$ is a solution to \eqref{eqn:rh:fund},
  including the requirement that $g_0 - w \in X^{k+\beta}$. Moreover,
  any solution $g \in \tilde X^{k+\beta}$ of $\Im\{a_2 g\} = 0$ is of
  the form $g = Cg_0$ for some real constant $C$. Since $C$ is also
  the coefficient of $w$ in the Fourier series for $g$, we conclude
  that $g - w \in X^{k+\beta}$ if and only if $C = 1$, and so $g_0$ is
  indeed the unique solution of \eqref{eqn:rh:fund}. This completes
  the proof of \ref{lem:rh:fund}.

  To prove \ref{lem:rh:inv}, we note that $L \maps X^{k+\beta} \to
  Y^{k-1+\beta}$ is simply the restriction $\mathcal
  S|_{X^{k+\beta}}$. Since $X^{k+\beta} \sub \tilde X^{k+\beta}$ has
  codimension 1 and $X^{k+\beta} \cap \ker \mathcal S = \varnothing$,
  the above Fredholm properties of $\mathcal S$ imply that $L$ is
  invertible. Taking the limit of \eqref{eqn:Fgen2} as $w$ approaches
  $\T$ and using the cancellation \eqref{eqn:simplifications} to drop
  one of the terms, we finally recover \eqref{eqn:rh:inv}.
\end{proof}

\subsection{Derivatives of the stream function}\label{sec:stream}
In this section we complete the proof of Lemma~\ref{lem:stream}
expressing the partial derivatives of the relative stream function
$\Psi$ in terms of the trace $\phi$ of the conformal mapping $\Phi$.
From potential theory, we know that the (non-relative) stream function
$\psi$ is given by
\begin{align*}
  \psi(z) = \frac 1{2\pi} \iint_D \log\abs{z-\zeta}\, d\zeta.
\end{align*}
Differentiating inside the integral, and using the complex form of
Green's theorem, we find that $\dell_z \psi(z)$ can be written as the
contour integral
\begin{align*}
  \dell_z \psi(z) = - \frac 1{4\pi} \int_D \frac 1{z-\zeta} \, d\zeta
  = \frac 1{8\pi i} \int_{\dell D} \frac{\bar z - \bar
  \zeta}{z - \zeta}\, d\zeta,
\end{align*}
provided $z \notin D$. Parametrizing $\dell D$ using $\zeta = \phi(\tau)$ and
setting $z = \Phi(w)$, we obtain
\begin{align*}
  \dell_z \psi(\Phi(w)) = \frac 1{8\pi i}
  \int_\T
  \frac{\bar{\Phi(w)}-\bar{\phi(\tau)}}
  {\Phi(w)-\phi(\tau)} \phi'(\tau)\, d\tau.
\end{align*}
Taking the limit as $w$ approaches $\T$, we conclude that
\begin{align*}
  (\dell_z \psi) \circ \phi = \tfrac 14 \Ca(\phi)\bar\phi.
\end{align*}
For higher derivatives of $\psi$, an additional step is needed to
obtain formulas valid for $w \in \T$. Arguing as before we find
\begin{align*}
  \dell_z^2\psi^-(z) &= 
  \frac 1{4\pi}\int_D \frac 1{(z-\zeta)^2}\, d\zeta
  = \frac 1{8\pi i} \int_{\dell D} \frac {\bar z - \bar \zeta}{(z-\zeta)^2}\, d\zeta
\end{align*}
for $z \notin D$,
but now additional care is needed to define the contour integral
on the right hand side for $z \in \dell D$. Rewriting the integral
using $\Phi$ as before, and integrating by parts, we find
\begin{align*}
  \dell_z^2\psi^-(\Phi(w))
  = \frac 1{8\pi i} \int_\T \frac {\bar {\Phi(w)} - \bar{\Phi(\tau)}}
  {(\Phi(w)-\Phi(\tau))^2}\, \Phi'(\tau) d\tau
  = \frac 1{8\pi i} \int_\T 
  \frac{ \frac{\bar{w^2\Phi'(w)}}{4\Phi'(w)}
  -  \frac{\bar{\tau^2\Phi'(\tau)}}{4\Phi'(\tau)}}
  {\Phi(w)-\Phi(\tau)} \Phi'(\tau)\, d\tau.
\end{align*}
Now we are free to take the limit as $w$ approaches $\T$, getting
  \begin{align*}
    \dell_z^2\psi^- \circ \phi
    =
    \tfrac 14 \Ca(\phi) F_2(\phi) 
    \qquad \text{where} \quad
    F_2(\phi) := \frac{w^2\phi'}{4\bar{\phi'}}.
  \end{align*}
Arguing similarly for $\dell^3_z \psi$ we find that
\begin{align*}
  (\dell_z^3\psi^-) \circ \phi  
  = \tfrac 14\Ca(\phi) F_3(\phi) ,
  \qquad \text{where} \quad
  F_3(\phi) := -\frac 14 \left(
  \frac{2w^3\phi'}{\bar{\phi'}^2}
  + \frac{w^4\phi''}{\bar{\phi'}^2}
  + \frac{w^2 \phi' \bar{\phi''}}{ \bar{\phi'}^3} \right).
\end{align*}
The formulas \eqref{eqn:Psizw} for the derivatives of $\Psi$ now
follow immediately from $\Psi = \psi - \frac\Omega 2 z\bar z$. 

\section*{Acknowledgments.}
Nader Masmoudi was partially supported by NSF-DMS grant 1716466.
Miles H.~Wheeler was partially supported by NSF-DMS grant 1400926.

\noindent \textsc{Courant Institute of Mathematical Sciences,
New York University,
251 Mercer Street,
New York, NY 10012,
USA},\\
\textit{E-mail address:} \texttt{zineb@cims.nyu.edu}\\
  
\noindent \textsc{Department of Mathematics, 
New York University in Abu Dhabi, 
Saadiyat Island, 
P.O. Box 129188, 
Abu Dhabi, 
United Arab Emirates};
\textsc{
Courant Institute of Mathematical Sciences, 
New York University, 
251 Mercer Street, 
New York, NY 10012, 
USA},\\
\textit{E-mail address:} \texttt{masmoudi@cims.nyu.edu}\\

\noindent \textsc{Faculty of Mathematics,
University of Vienna,
Oskar-Morgenstern-Platz 1,
1090 Wien,
Austria},\\
\textit{E-mail address:} \texttt{miles.wheeler@univie.ac.at}

\end{document}